\documentclass[11pt,reqno]{amsart}

\usepackage{epsfig}
\usepackage{subfigure}
\usepackage{verbatim}
\usepackage[dvipsnames,usenames]{color} 
\usepackage{amssymb}
\usepackage{amsfonts}
\usepackage{amsbsy}
\usepackage{cite}
\usepackage[cmyk]{xcolor}
\usepackage{setspace}

\usepackage{graphicx,subfigure,boxedminipage}
\usepackage{amsmath}
\usepackage[english]{babel}
\usepackage{color}
\usepackage{tikz}
\usepackage{amsfonts}
\usepackage{graphics,verbatim}
\usepackage{soul}

\newcommand{\M}{{\mathcal M}}

\newcommand{\R}{{\mathbb R}}

\renewcommand{\eqref}[1]{(\ref{#1})}

\newtheorem{thm}{Theorem}[section]
\newtheorem{theo}[thm]{Theorem}

\newtheorem{coro}[thm]{Corollary}

\newtheorem{lem}[thm]{Lemma}

\newtheorem{prop}[thm]{Proposition}

\theoremstyle{definition}
\newtheorem{exam}{Example}[section]
\newtheorem{defi}{Definition}[section]
\newtheorem{rmk}[thm]{Remark}

\numberwithin{equation}{section}

\renewcommand{\theequation}{\thesection.\arabic{equation}}

\def\vint{\mathop{\mathchoice%
         {\setbox0\hbox{$\displaystyle\intop$}\kern 0.22\wd0%
          \vcenter{\hrule width 0.6\wd0}\kern -0.82\wd0}%
         {\setbox0\hbox{$\textstyle\intop$}\kern 0.2\wd0%
          \vcenter{\hrule width 0.6\wd0}\kern -0.8\wd0}%
         {\setbox0\hbox{$\scriptstyle\intop$}\kern 0.2\wd0%
          \vcenter{\hrule width 0.6\wd0}\kern -0.8\wd0}%
         {\setbox0\hbox{$\scriptscriptstyle\intop$}\kern 0.2\wd0%
          \vcenter{\hrule width 0.6\wd0}\kern -0.8\wd0}}%
         \mathopen{}\int}

\begin{document}

\title[Nodal sets and continuity of eigenfunctions on Riemannian manifolds]
{Nodal sets and continuity of eigenfunctions of Kre$\mathbf{\breve{{\i}}}$n-Feller operators on Riemannian manifolds}

\author[S.-M. Ngai]{Sze-Man Ngai}
\address{Beijing Institute of Mathematical Science and Applications Key Laboratory of High Performance Computing and Stochastic Information Processing (HPCSIP) (Ministry of Education of China), College of
Mathematics and Statistics, Hunan Normal University, Changsha, Hunan 410081, China.}
\email{ngai@bimsa.cn}

\author[W.-Q. Zhao]{Wen-Quan Zhao}
\address{Key Laboratory of High Performance Computing and Stochastic Information
Processing (HPCSIP) (Ministry of Education of China), College of
Mathematics and Statistics, Hunan Normal University, Changsha, Hunan
410081, China.}
\email{zhaowq1008@hunnu.edu.cn}

\thanks{The authors are supported in part by the National Natural Science Foundation of China, grants 12271156, and Construct Program of Key Discipline in Hunan Province.}

\begin{abstract} Let $d\geq1$, $\Omega$ be a bounded domain of a smooth complete Riemannian $d$-manifold $M$, and $\mu$ be a positive finite Borel measure with compact support contained in $\overline{\Omega}$. We prove the Courant nodal domain theorem for the eigenfunctions of the Kre\u{\i}n-Feller operator $\Delta_{\mu}$ under the assumption that such eigenfunctions are continuous on $\overline{\Omega}$. For $d\geq2$, we prove that on a bounded domain $\Omega\subseteq M$ with smooth boundary and on which Green's function of the Laplace-Beltrami operator exists, the eigenfunctions of $\Delta_{\mu}$ are continuous on $\Omega$. We also prove that if $M$ is compact and $\partial M=\emptyset$, then the eigenfunctions of $\Delta_{\mu}$ are continuous on $M$.
\end{abstract}

\date{\today}
\maketitle

\section{Introduction}
Let $M$ be a $d$-dimensional compact Riemannian manifold with boundary $\partial M$ and $\Delta$ be the Laplace-Beltrami operator defined on $M$. It is well known that the eigenvalues of $\Delta$ can be ordered as
 $$0\leq\lambda_{1}\leq\lambda_{2}\leq\cdots.$$
Moreover, the zero eigenvalue is taken only in the case $\partial M=\emptyset$ and eigenfunctions associated with the eigenvalue $0$ are constants (see, e.g., \cite[Section 3]{Schoen-Yau_1994}). Yau \cite{Yau_1982} conjectured that the nodal set of a $\lambda$-eigenfunction $u_{\lambda}$ on $M$ satisfies:
$$c\lambda^{1/2}\leq\mathcal{H}^{d-1}(\{x\in M:u_{\lambda}(x)=0\})\leq C\lambda^{1/2},$$
where $c$ and $C$ are positive constants and $\mathcal{H}^{d-1}$ is the $(d-1)$-dimensional Hausdorff measure. This conjecture has been studied extensively (see \cite{Donnelly-Fefferman_1988,Donnelly-Fefferman_1990,Logunov_20181,Logunov_20182,Logunov-Malinnikova_2018} and references therein).

A well-known property of the nodal set of an eigenfunction of $\Delta$ is the Courant nodal domain theorem, which states that for a bounded domain $\widetilde{\Omega}\subseteq\mathbb{R}^{d}$, the nodal set of a $\lambda_{i}$-eigenfunction $u_{\lambda_{i}}$ ($i\geq1$) divides $\widetilde{\Omega}$ into at most $i$ subdomains, i.e., the number of nodal domains of $u_{\lambda_{i}}$ is at most $i$ (see, e.g., \cite{Courant-Hilbert_1953}). Results concerning the Courant nodal domain theorem on a bounded domain of $\mathbb{R}^{d}$ can be found in \cite{Alessandrini_1994,Alessandrini_1998,Gladwell-Zhu_2002} and references therein. The Courant nodal domain theorem on a smooth compact Riemannian manifold $M$ with boundary $\partial M$ is slightly more complicated than that on $\mathbb{R}^{d}$ and can be stated as (see, e.g., \cite[Theorem 6.2]{Schoen-Yau_1994}) follows:

\begin{enumerate}
\item[(a)] If $\partial M\neq\emptyset$, then the number of nodal domains of $u_{\lambda_{i}}$ is at most $i$;
\item[(b)] If $\partial M=\emptyset$, then the number of nodal domains of $u_{\lambda_{i}}$ is at most $i+1$.
\end{enumerate}
Cheng \cite{Cheng_1976} showed that $(a)$ above still holds on a compact domain $\Omega$ of $M$ with $\partial\Omega\neq\emptyset$. The Courant nodal domain theorem of eigenfunctions of other operators, such as Dirac operators on manifolds, Schrodinger operators on graph have also been studied (see, e.g., \cite{Bar_1997,Bartsch-Liu-Weth_2004,Davies-Gladwell-Leydold-Stadler_2001}). In this paper, we study the Courant nodal domain theorem for Kre\u{\i}n-Feller operators defined on Riemannian manifolds.

Kre\u{\i}n-Feller operators are introduced by Kre\u{\i}n and Feller in \cite{Feller_1957,Krein_1952} and have been studied extensively.  Ngai {\em et al.} \cite{Ngai-Zhang-Zhao_2024} studied the Courant nodal domain theorem and the continuity of eigenfunctions of Kre\u{\i}n-Feller operators defined on a domain of $\mathbb{R}^{d}$. Ngai and Ouyang \cite{Ngai-Ouyang_2024} studied the spectrum properties of Kre\u{\i}n-Feller operators on a bounded domain of a smooth complete Riemannain manifold. For more work associated with the Kre\u{\i}n-Feller operators, the reader is referred to  \cite{Deng-Ngai_2021,Hu-Lau-Ngai_2006,Kessebohmer-Niemann_2022,Kessebohmer-Niemann_2022A} and references therein.

In this paper, we first prove the nodal domain theorem for Kre\u{\i}n-Feller operators on a bounded domain $\Omega$ of a complete smooth Riemannian manifold $M$, by assuming the continuity of the eigenfunctions. It is worth noting that we allow the boundary of $\Omega$ to be empty; in this case, the domain of $\Delta_{\mu}$ becomes a compact Riemannian manifold. Hence our theorem is consistent with Cheng's in \cite{Cheng_1976} for the Laplace-Beltrami operator. The main difficulty to prove the nodal domain theorem is to prove the maximum principle of $\mu$-subharmonic functions (see Theorem \ref{thmmax}). We use normal coordinate charts constructed in \cite{Ngai-Ouyang_2024} and the maximum principle of a $\mu$-subharmonic function on a bounded domain of $\mathbb{R}^{d}$ to prove the maximum principle on $\Omega$.

The second part of this paper is to study the continuity of eigenfunctions of $\Delta_{\mu}$. The motivation comes from the definition of the node set (see (\ref{nodalset})). Continuity of eigenfunctions is essential in studying many interesting problems, such as Courant's nodal domain theorem, nodal lines of a second eigenfunction \cite{Lin_1987}, and the generalization of Yau's conjecture to Kre\u{\i}n-Feller operators. We define the Green operator $G_{\mu}$ (see (\ref{go})) and prove that $(-\Delta_{\mu})^{-1}$, the inverse of $-\Delta_{\mu}$, equals the sum of $G_{\mu}$ and a harmonic operator (see Theorem \ref{thGmu}) on a bounded domain $\Omega\subseteq M$ that has a smooth boundary, and on which the (Dirichlet) Green's function exists. Inspired by \cite[Chapter 3]{Ponce_2006}, we study weak solutions of the Dirichlet problem with density measure $\mu$ and prove Theorem \ref{thGmu}. Relating the uniqueness of the weak solution in the Sobolev space with the lower $L^{\infty}$-dimension of the measure $\mu$ enables us to prove the theorem. We use Theorem \ref{thGmu} and the methods in \cite{Ngai-Zhang-Zhao_2024} to prove the continuity of the eigenfunctions on $\Omega$. Finally, on a compact connected smooth Riemannian manifold $M$ with $\partial M=\emptyset$, by using a similar method, we prove that the eigenfunctions of $\Delta_{\mu}$ are continuous.

This paper is organized as follows. In Section 2, we summarize the definition of $\Delta_{\mu}$ and give some notation and definitions. We also state the main results of this paper. In Section 3, we prove the maximum principle for $\mu$-subharmonic functions. In Section 4, we study the spectral properties of $\Delta_{\mu}$. In Section 5, we study the weak solution of a Dirichlet problem with density $\mu$. Moreover, in Section 6, we study the continuity of eigenfunctions of $\Delta_{\mu}$ on a bounded domain $\Omega$. In Section 7, we study the continuity of eigenfunctions of $\Delta_{\mu}$ on a compact connected smooth Riemannian manifold $M$ by assuming that $\partial M=\emptyset$. In Section 8, we study $\Delta_{\mu}$ defined on a domain of a Riemann surface which is conformally equivalent to a bounded domain of $\mathbb{R}^{2}$, and give an example of a continuous eigenfunction.

\section{Preliminaries and main relsults}
Let $d\geq1$, $(M,g)$ be a complete smooth Riemannian $d$-manifold with Riemannian metric $g$ and let $\nu$ be the Riemannian volume measure on $M$. Throughout this paper, we assume $M$ is oriented. Let $\Omega\subseteq M$ be a bounded domain (open and connected). Let $\overline{\Omega}$ and $\partial\Omega:=\overline{\Omega}\setminus\Omega$ be the closure and boundary of $\Omega$ respectively. Let $W^{1,2}(\Omega)$ be the Sobolev space equipped with the inner product
$$\langle u,v\rangle_{W^{1,2}(\Omega)}:=\int_{\Omega}uv\,d\nu+\int_{\Omega}\langle\nabla u,\nabla v\rangle\,d\nu,$$
where $\langle\cdot,\cdot\rangle=g(\cdot,\cdot)$. Let $C(\Omega)$ and $C_{c}^{\infty}(\Omega)$ be, respectively, the spaces of  continuous functions and smooth functions on $\Omega$ with compact support.  Let $W_{0}^{1,2}(\Omega)$ be the closure of $C_{c}^{\infty}(\Omega)$ in $\|\cdot\|_{W^{1,2}(\Omega)}$.

When $\partial \Omega=\emptyset$, the Laplace-Beltrami operator is self-adjoint and has discrete eigenvalues $0=\lambda_{0}<\lambda_{1}\leq\lambda_{2}\leq\cdots$. Moreover, the associated eigenfunctions $\phi_{i}$, $i\geq0$, satisfy
\begin{equation}\label{DP}
\Delta \phi_{i}=-\lambda_{i}u_{i},\quad \phi_{i}\not\equiv0.
\end{equation}
Moreover, there exists an orthonormal bases $\{\phi_{i}\}_{i=1}^{\infty}$ of $W_{0}^{1,2}(\Omega)$ consisting of eigenfunctions. When $\partial \Omega\not=\emptyset$, in order to ensure that the Laplace-Beltrami operator is self-adjoint, we consider the Dirichlet problem:
\begin{equation}\label{DP1}
\begin{cases}~-\Delta u(x)=\lambda u(x),&x\in \Omega,\\~ u(x)=0,&x\in\partial \Omega.\end{cases}
\end{equation}
In this case, ${\rm dom}(\Delta)=W_{0}^{1,2}(\Omega)$ and the (Dirichlet) eigenvalues of (\ref{DP1}) satisfy $0<\lambda_{1}\leq\lambda_{2}\leq\cdots$. Moreover, there exists an orthonormal basis $\{\phi_{i}\}_{i=1}^{\infty}$ of $W_{0}^{1,2}(\Omega)$ consisting of eigenfunctions of $-\Delta$.

We briefly describe the definition of a Kre\u{\i}n-Feller operator for the case $\partial\Omega\neq\emptyset$ and omit the details. The case $\partial\Omega=\emptyset$ is similar. The reader is referred to \cite{Hu-Lau-Ngai_2006,Ngai-Ouyang_2024,Ngai-Ouyang_2024a} for details. Since $\partial\Omega\neq\emptyset$, the {\em Poincar\'{e} inequality} holds (see, e.g., \cite[Scetion 3]{Schoen-Yau_1994}): there exists a positive constant $C\in(0,\lambda_{1}]$ such that for any $f\in W_{0}^{1,2}(\Omega)$,
\begin{equation}\label{eq1}
\int_{\Omega}|f|^{2}\,d\nu\leq C^{-1}\int_{\Omega}|\nabla f|^{2}\,d\nu.
\end{equation}
Note that the Poincar\'{e} inequality (\ref{eq1}) implies that $W_{0}^{1,2}(\Omega)$ admits an equivalent inner product defined by
$$\langle u,v\rangle_{W_{0}^{1,2}(\Omega)}:=\int_{\Omega}\langle\nabla u,\nabla v\rangle\,d\nu.$$
Let $\mu$ be a positive finite Borel measure on $M$ with {\rm supp}$(\mu)\subseteq\overline{\Omega}$ and $\mu(\Omega)>0$. Assume $\mu$ satisfies the following {\em Poincar\'e inequality for measures} on $\Omega$:

(MPI) There exists a constant $C>0$ such that for all $v\in C_{c}^{\infty}(\Omega)$,
\begin{equation}\label{MPI}
\int_{\Omega}|u|^{2}\,d\mu\leq C\int_{\Omega}|\nabla u|^{2}\,d\nu.
\end{equation}
(MPI) implies that there exists a closed subspace $\mathcal{F}$ of $W_{0}^{1,2}(\Omega)$ such that the embedding $\mathcal{F}\hookrightarrow L^{2}(\Omega,\mu)$ is injective and dense; moreover, the nonnegative bilinear form $\mathcal{E}(\cdot,\cdot)$ defined on $L^{2}(\Omega,\mu)$ by
\begin{equation}\label{form}
\mathcal{E}(u,v):=\int_{\Omega}\langle\nabla u,\nabla v\rangle\,d\nu
\end{equation}
is closed with {\rm dom}$(\mathcal{E})=\mathcal{F}$ (see \cite[Proposition 3.2]{Ngai-Ouyang_2024}).
Therefore, there exists a nonnegative self-adjoint operator $-\Delta_{\mu}$ such that ${\rm dom}(\mathcal{E})={\rm dom}((-\Delta_{\mu})^{1/2})$ and
$$\mathcal{E}(u,v)=\big\langle(-\Delta_{\mu})^{1/2}u ,~(-\Delta_{\mu})^{1/2}v\big\rangle_{L^{2}(\Omega,\mu)},\qquad \text{for~all~$u,v\in\rm{dom}(\mathcal{E})$}$$
(see \cite{Davies_1995}), where $\langle\cdot,\cdot\rangle_{L^{2}(\Omega,\mu)}$ is the inner product on $L^{2}(\Omega,\mu)$.
 We call the above $\Delta_{\mu}$ the {\em(Dirichlet) Laplacian} with respect to $\mu$ or the {\em Kre\u{\i}n-Feller operator} defined by $\mu$.
It follows from \cite[Proposition 3.3]{Ngai-Ouyang_2024} that
$u\in\rm{dom}(\Delta_{\mu})$ and $-\Delta_{\mu}u=f$ if and only if $-\Delta u=f\,d\mu$ in the sense of distribution, i.e.,
\begin{equation}
\int_{\Omega}\langle\nabla u,\nabla v\rangle\,d\nu=\int_{\Omega}fv\,d\mu, \qquad~\text{for all}~v\in C_{c}^{\infty}(\Omega).
\end{equation}

Recall that the {\em lower $L^{\infty}$-dimension} of $\mu$ is defined as
\begin{equation}\label{dim}
\underline{\rm{dim}}_{\infty}(\mu):=\liminf_{\delta\to0^{+}}\frac{\ln(\sup_{x}\mu(B^{M}_{\delta}(x)))}{\ln{\delta}},
\end{equation}
where $B^{M}_{\delta}(x)$ is the ball with radius $\delta$ and center $x\in{\rm supp}(\mu)$, and the supremum is taken over all $x\in{\rm supp}(\mu)$ (see \cite{Strichartz_1993} for details). It is known that if $\underline{\dim}_{\infty}(\mu)>d-2$, then (MPI) holds on $\Omega$ (\cite[Theorem 2.1]{Ngai-Ouyang_2024}). Throughout this paper, we assume the measure $\mu$ satisfies $\underline{\dim}_{\infty}(\mu)>d-2$ to ensure that the operator $-\Delta_{\mu}$ is well defined.

For a bounded domain $\Omega\subseteq M$, if $\partial \Omega\not=\emptyset$, we consider eigenvalues and eigenfunctions associated with the following Dirichlet problem:
\begin{equation}\label{eigen}
\begin{cases}~-\Delta_{\mu}u(x)=\lambda u(x),&x\in \Omega,\\~ u(x)=0,&x\in\partial \Omega.\end{cases}
\end{equation}
If $\partial \Omega=\emptyset$, the associated eigenvalues and eigenfunctions are the solutions of
\begin{equation}\label{eigen2}
-\Delta_{\mu}u(x)=\lambda u(x),\quad x\in \Omega.
\end{equation}
It is shown in \cite[Theorem 2.2]{Ngai-Ouyang_2024} that there exists an orthonormal basis $\{\phi_{n}\}_{n=1}^{\infty}$ of $L^{2}(\Omega,\mu)$ consisting of (Dirichlet) eigenfunctions of $-\Delta_{\mu}$.
The eigenvalues $\{\lambda_{n}\}_{n=1}^{\infty}$ satisfy $0\leq\lambda_{1}\leq\lambda_{2}\leq\cdots.$ Throughout this paper, we let $\lambda_{0}:=0$ be the eigenvalue $0$ of $-\Delta_{\mu}$ if it exists.

Note that when $d\geq2$, the spaces $W_{0}^{1,2}(\Omega)$ and ${\rm dom}(\Delta_{\mu})$ may contain discontinuous functions. Suppose an eigenfunction $u$ of $-\Delta_{\mu}$ belongs to $C(\Omega)$. We define the {\em nodal set} of $u$ as
\begin{equation}\label{nodalset}
\mathcal{Z}_{\mu}(u):=\{x\in \Omega: u(x)=0\}.
\end{equation}
Write $\Omega\setminus\mathcal{Z}_{\mu}(u)=\bigcup_{i}\Omega_{i}$, where $\Omega_{i}$ is open and connected. Each $\Omega_{i}$ is called a {\em nodal domain} of $u$. Note that $\Omega_{i}\cap\Omega_{j}=\emptyset$ for any $i\neq j$. Let
\begin{equation}\label{nodalD}
\mathcal{N}_{\mu}(u):=\text{the~number~of~nodal~domains~of}~u.
\end{equation}

Let $M$ be a complete smooth manifold and $\Omega\subseteq M$ be a bounded domain. Let $\mu$ be a positive finite Borel measure on $M$ with ${\rm supp}(\mu)\subseteq \overline{\Omega}$ and $\mu(\Omega)>0$. We prove the maximum principle for $\mu$-subharmonic functions (see Section 3). We use the Rayleigh quotient and the maximum principle of $\mu$-subharmonic functions to prove the following Courant nodal domain theorem.

\begin{theo}\label{Courant2}
Let $d\geq1$, $M$ be a complete smooth Riemannian $d$-manifold and $\Omega\subseteq M$ be a bounded domain with boundary $\partial\Omega$. Let $\mu$ be a positive finite Borel measure on $M$
satisfying ${\rm supp}(\mu)\subseteq\overline{\Omega}$, $\mu(\Omega)>0$ and $\underline{\dim}_{\infty}(\mu)>d-2$.
Assume that the eigenvalues $\{\lambda_{n}\}_{n=1}^{\infty}$ of $-\Delta_{\mu}$ are arranged in an increasing order and
let $u_{n}(x)$ be a $\lambda_{n}$-eigenfunction. Assume that $u_{n}$ is continuous on $\overline{\Omega}$ and that $\mathcal{N}_{\mu}(u_{n})$ is defined as in (\ref{nodalD}).
\begin{enumerate}
\item[(a)]~If $\partial \Omega\neq\emptyset$, then $\mathcal{N}_{\mu}(u_{n})\leq n$.

\item[(b)]~If $\partial \Omega=\emptyset$, then $\mathcal{N}_{\mu}(u_{n})\leq n+1$.
\end{enumerate}
\end{theo}

As the definition of nodal sets makes sense only if the corresponding eigenfunction is continuous, we study the continuity of eigenfunctions of $-\Delta_{\mu}$ on $\Omega$. Note that when $d=1$, the Sobolev space $W_{0}^{1,2}(\Omega)$ is embedded in $C(\overline{\Omega})$. Hence, the eigenfunctions of $-\Delta_{\mu}$ are continuous on $\Omega$. We need only consider the case $d\geq2$.

Let $d\geq2$ and $M$ be a smooth complete Riemannian $d$-manifold. Let $\Omega\subseteq M$ be a bounded domain with smooth boundary. In view of the Hopf-Rinow theorem (see e.g. \cite[Theorem 5.7.1]{Petersen_2006}), $\overline{\Omega}$ is compact. It is known that if $\overline{\Omega}$ is a Riemannian manifold with a nonempty smooth boundary, then the (Dirichlet) Green function (see Section 5 or \cite{Grigoryan_1999,Li-Tam_1987} for the definition) of the Laplace-Beltrami operator exists on $\Omega$ (see, e.g., \cite[Theorem 4.17]{Aubin_1982}). Examples of such domains include a bounded domain $\Omega$ of a complete noncompact smooth Riemannian manifold with $\partial\Omega$ being smooth, and a bounded domain $\Omega$ of a compact smooth Riemannian manifold $M$ satisfying $M\setminus\overline{\Omega}\neq\emptyset$. Let $\mu$ be a positive finite Borel measure on $M$ with ${\rm supp}(\mu)\subseteq \overline{\Omega}$ and $\mu(\Omega)>0$. Assume $\underline{\dim}_{\infty}(\mu)>d-2$. We first consider the inverse operator of $-\Delta_{\mu}$ on $L^{2}(\Omega,\mu)$. We define a bounded linear operator $G_{\mu}$ (definition see Section 5) on $L^{2}(\Omega,\mu)$, called the {\em Green operator}. We prove that the inverse operator of $-\Delta_{\mu}$ defined on $\Omega$ can be decomposed as the sum of the Green operator and a harmonic operator.

We apply the Green operator to study the continuity of eigenfunctions. By using the properties of Green's function (see Proposition \ref{pg}) and the methods in \cite[Theorem 1.2]{Ngai-Zhang-Zhao_2024}, we obtain the following theorem.

\begin{theo}\label{thcon}
Let $d\geq2$, $M$ be a smooth complete Riemannian $d$-manifold and $\Omega\subseteq M$ be a bounded domain with smooth boundary and on which Green's function exists. Let $\mu$ be a positive finite Borel measure on $M$ with {\rm supp}$(\mu)\subseteq\overline{\Omega}$ and $\mu(\Omega)>0$. Assume $\underline{\dim}_{\infty}(\mu)>d-2$ and $\underline{\dim}_{\infty}(\nu)>d-2$. Then the eigenfunctions of $-\Delta_{\mu}$ are continuous on $\Omega$.
\end{theo}

As in the proof of Theorem \ref{thcon}, we define Green's operator $G_{\mu}$ on a smooth compact connected Riemannian manifold and prove that the restriction of $G_{\mu}$ on a subspace of $L^{2}(M,\mu)$ is the inverse operator of the restriction of $\Delta_{\mu}$ on a subspace of {\rm dom}$(\Delta_{\mu})$ (see Theorem \ref{OG}). We study the continuity of eigenfunctions of $\Delta_{\mu}$ and obtain the following theorem.
\begin{thm}\label{CCT}
Let $d\geq2$ and $M$ be a compact connected smooth Riemannian $d$-manifold with $\partial M=\emptyset$. Let $\mu$ be a positive finite Borel measure on $M$. Assume that $\underline{\dim}_{\infty}(\mu)>d-2$. Then the eigenfunctions of $\Delta_{\mu}$ are continuous on $M$.
\end{thm}

\section{Maximum Principle}
Let $M$ be a complete smooth Riemannian manifold and $\Omega\subseteq M$ be a bounded domain. Let $\mu$ be a finite positive Borel measure defined on $M$ with {\rm supp}$(\mu)\subseteq\overline{\Omega}$ and $\mu(\Omega)>0$. In this section, we study the maximum principle for $\mu$-subharmonic functions (see Definition \ref{musub}) on $\Omega$.

We first introduce normal coordinate charts of $M$. Based on these charts, we can map a geodesic ball in $M$ with sufficiently small radius to $\mathbb{R}^{d}$ by a diffeomorphism. The following method of constructing normal coordinate charts is from \cite{Ngai-Ouyang_2024}. The reason we need to use the complicated construction of normal coordinate charts $\{(U_{i},\varphi_{i})\}_{i}$ in \cite{Ngai-Ouyang_2024} instead of the standard one is that we can choose the size of $U_{i}$ so that $\overline{\Omega}$ can be covered by a family of subset of $U_{i}$.

Let $\rho(x,y)$ be the Riemannian distance between points $x,y\in M$. Let $p\in M$ and $T_{p}M$ be the tangent space of $M$ at $p$. For $r>0$, let
\begin{align*}
&B_{r}(x):=\{y\in\mathbb{R}^{d}: |x-y|<r\},\quad x\in\mathbb{R}^{d},\\
&B^{M}_{r}(p):=\{q\in M: \rho(x,y)<r\},\quad p\in M,\\
&B^{T_{p}M}_{r}(0):=\{\zeta\in T_{p}(M):|\zeta|<r\}.
\end{align*}
For any $p\in M$ and every orthonormal basis $(b_{i})$ of $T_{p}M$, there is a basis isomorphism
\begin{equation}\label{Ep}
E_{p}:T_{p}M\to\mathbb{R}^{d}.
\end{equation}

Let $i\in\mathbb{N}^{+}$. For each $p_{i}\in \overline{\Omega}$, let $\epsilon_{i}:=\epsilon_{p_{i}}\in(0,{\rm inj}(\overline{\Omega}))$. Since $\overline{\Omega}$ is compact, there exists a finite open cover $\{B^{M}_{\epsilon_{i}}(p_{i})\}_{i=1}^{N}$ of  $\overline{\Omega}$, where $p_{i}\in\overline{\Omega}$ and each $B^{M}_{\epsilon_{i}}(p_{i})$ is a geodesic ball. For each $i=1,\ldots,N$, the exponential map $\exp_{p_{i}}: B^{T_{p_{i}}M}_{\epsilon_{i}}(0)\to B^{M}_{\epsilon_{i}}(p_{i})$ is a diffeomorphism. Let each $E_{p_{i}}:T_{p_{i}}M\to \mathbb{R}^{d}$ be defined as in (\ref{Ep}) and denote $B_{\epsilon_{i}}(0):=E_{p_{i}}(B^{T_{p_{i}}}_{\epsilon_{i}}(0))$. Let $S_{i}:B_{\epsilon_{i}}(0)\to B_{\epsilon_{i}}(z_{i})$ such that the collection $B_{\epsilon_{i}}(z_{i})$, where $z_{i}\in\mathbb{R}^{d}$, are disjoint. We now define a family of normal coordinate maps:
\begin{equation}\label{NCM}
\varphi_{i}:=S_{i}\circ E_{p_{i}}\circ\exp^{-1}_{p_{i}}: B^{M}_{\epsilon_{i}}(p_{i})\to B_{\epsilon_{i}}(z_{i}),\quad i=1,\ldots,N,
\end{equation}
where $\varphi_{i}(p_{i})=z_{i}$ and the sets $\varphi_{i}(B^{M}_{\epsilon_{i}}(p_{i}))$ are disjoint. For convenience, we let $U_{i}:=B^{M}_{\epsilon_{i}}(p_{i})$ and $\widetilde{U}_{i}:=B_{\epsilon_{i}}(z_{i})$. Therefore, the system of normal coordinate charts of $\Omega$ is given by $\{(U_{i},\varphi_{i})\}_{i=1}^{N}$, where each $\varphi_{i}$ is a diffeomorphism.

Let $\mu$ be a positive finite Borel measure with supp$(\mu)\subseteq\overline{\Omega}$. Let $\{(U_{i},\varphi_{i})\}_{i=1}^{N}$ be normal coordinate charts defined as above with $\varphi_{i}$ satisfying (\ref{NCM}). Then for each $i=1,\ldots N$, $\varphi_{i}$ induces a measure
\begin{equation}\label{m1}
\widetilde{\mu}_{i}:=\mu\circ\varphi_{i}^{-1}
\end{equation}
on $\widetilde{U}_{i}$ and each $\widetilde{\mu}_{i}$ satisfies supp$(\widetilde{\mu}_{i})\subseteq\overline{\widetilde{U_{i}}}$. Let $\tau$ be a positive finite Borel measure with compact support on $\mathbb{R}^{d}$. Assume ${\rm supp}(\tau)\subseteq\overline{\widetilde{U}}$ for some bounded domain $\widetilde{U}\subseteq\mathbb{R}^{d}$. We say that $\tau$ satisfies (MPI) on $\widetilde{U}$ if there exists a constant $C>0$ such that for any $\widetilde{u}\in C_{c}^{\infty}(\widetilde{U})$,
$$\int_{\widetilde{U}}|\widetilde{u}|^{2}\,d\tau\leq C\int_{\widetilde{U}}|\nabla\widetilde{u}|^{2}\,dx.$$

The proof of the following proposition is straightforward and is omitted.
\begin{prop}\label{mpi}
Let $d$, $M$ and $\Omega$ be as in Theorem \ref{Courant2}. Let $\mu$ be a positive finite Borel measure on $M$ with {\rm supp}$(\mu)\subseteq \overline{\Omega}$ and $\mu(\Omega)>0$. Let $\widetilde{\mu}_{i}$ be defined as (\ref{m1}), where $i=1,\ldots,N$. Assume that $\mu$ satisfies (MPI) on $M$. Then $\widetilde{\mu}_{i}$ satisfies (MPI) on $\widetilde{U_{i}}$.
\end{prop}

The following lemma can be derived by modifying the proofs of \cite[Proposition 3.2 and Theorem 3.4]{Ngai-Ouyang_2024}; we will omit some details.
\begin{lem}\label{Lemma1}
Let $V\subseteq\R^{d}$ be a bounded domain and $\omega$ be a positive finite Borel measure on $V$ with {\rm supp}$(\omega)\subseteq \overline{V}$ and $\omega(V)>0$. Let $u\in W^{1,2}(V)\cap C(\overline{V})$. Assume $\omega$ satisfies (MPI) on $V$. If for all $v\in C_{c}^{\infty}(V)$, there exists an $\omega$-a.e. nonnegative function $f\in L^{2}(V,\omega)$ such that
\begin{equation}\label{max}
\int_{V}u\Delta v\,dx=\int_{V}f v\,d\omega,
\end{equation}
then $u$ cannot attain its maximum in $V$, unless $u$ is a constant.
\end{lem}
\begin{proof}
By Tietze's extension theorem (see, e.g., \cite[Theorem 35.1]{Munkres_2000}), there exists some $\widehat{V}\subset\subset\mathbb{R}^{d}$ with $V\subseteq\widehat{V}$ and some $\overline{u}\in C(\overline{\widehat{V}})$ with {\rm supp}$(\overline{u})\subseteq\overline{\widehat{V}}$ such that $\overline{u}=u$ on $V$. For any $\delta>0$, let
$$\overline{u}^{\delta}:=\eta_{\delta}\ast\overline{u},$$
where $\eta_{\delta}\geq0$ is a mollifier. Then $\overline{u}^{\delta}\in C^{\infty}(\widehat{V}_{\epsilon})$ and $\overline{u}^{\delta}\to u$ uniformly on $V$ (see e.g.,\cite[\S C5]{Evans_2010}). Let $z\in V$. Then for any $0<r<{\rm dist}(z,\partial V)$, using an argument as that in the proof of \cite[Proposition 3.2]{Ngai-Zhang-Zhao_2024}, one can prove that
\begin{equation}\label{eqlim}
\lim_{\delta\to0}\int_{B_{r}(z)}\Delta(\overline{u}^{\delta}|_{V})(x)\,dx=\int_{B_{r}(z)}f(x)\,d\omega.
\end{equation}
Using (\ref{eqlim}) and an argument in the proof \cite[Theorem 3.4]{Ngai-Zhang-Zhao_2024} completes the proof (see \cite[Theorem 3.4]{Ngai-Zhang-Zhao_2024} for details).
\end{proof}
We will prove the maximum principle for a continuous $\mu$-subharmonic function. The following definition of $\mu$-subharmonic functions can be found in \cite[Definition 3.1]{Ngai-Zhang-Zhao_2024}.
\begin{defi}\label{musub}
We call $u\in{\rm dom}(-\Delta_{\mu})$ a {\em $\mu$-subharmonic function} if $\Delta_{\mu}u\geq0$ ($\mu$-a.e.). Call $u\in{\rm dom}(-\Delta_{\mu})$ a {\em$\mu$-superharmonic function} if $\Delta_{\mu}u\leq0$ ($\mu$-a.e.).
\end{defi}

The following maximum principle theorem generalize \cite[Theorem 3.4]{Ngai-Zhang-Zhao_2024} to a bounded domain $\Omega\subseteq M$.
\begin{thm}\label{thmmax}
Let $d$, $M$, $\Omega$, and $\mu$ be as in Proposition \ref{mpi} and assume that $\mu$ satisfies (MPI) on $M$. If $u\in C(\overline{\Omega})$ is a nonconstant $\mu$-subharmonic function, then $u$ cannot attain its maximum value in $\Omega$.
\end{thm}
\begin{proof}
For $i\in \mathbb{N}^{+}$ and $p_{i}\in\overline{\Omega}$, let $\epsilon_{i}:=\epsilon_{p_{i}}\in(0,{\rm inj}(\overline{\Omega}))$. Since $\overline{\Omega}$ is compact, there exists normal coordinate charts $\{(B^{M}_{\epsilon_{i}}(p_{i}),\varphi_{i})\}_{i=1}^{N}$ such that $\overline{\Omega}\subseteq\bigcup_{i=1}^{N}B^{M}_{\epsilon_{i}/2}(p_{i})$, where $\varphi_{i}$ is defined as in (\ref{NCM}). The reason we choose these coordinate charts is to ensure that the composition of a continuous map on $\overline{B^{M}_{\epsilon_{i}/2}(p_{i})}$ and $\varphi|_{\overline{B^{M}_{\epsilon_{i}/2}(p_{i})}}$ is continuous on $\overline{\varphi(B_{\epsilon_{i}/2}(z_{i}))}$. For convenience, we let $U_{i}:=B^{M}_{\epsilon_{i}/2}(p_{i})$, $i=1,\ldots,N$. Suppose $u(x)$ attains its maximum at some $p_{0}\in \Omega$. Then there would exist some integer $i_{0}\in\{1,\ldots,N\}$ such that $p_{0}\in U_{i_{0}}$. Without loss of generality, we assume $U_{i_{0}}\subseteq\Omega$. In fact, if $U_{i_{0}}\cap\Omega^{c}\neq\emptyset$, where $\Omega^{c}:=M\setminus\Omega$, the same argument works on $U_{i_{0}}\cap\Omega$. Let $v\in C_{c}^{\infty}(U_{i_{0}})\subseteq C_{c}^{\infty}(\Omega)$. For convenience, we let $f:=\Delta_{\mu}u$ and $\widetilde{U}_{i_{0}}:=\varphi_{i_{0}}(U_{i_{0}})$. Then
$$\int_{U_{0}}u\Delta v\,d\nu=\int_{U_{0}}f\cdot v\,d\mu,$$
which implies that
\begin{equation}\label{eqdef}
\int_{\widetilde{U}_{i_{0}}}u\circ\varphi_{i_{0}}^{-1}\Delta(v\circ\varphi_{i_{0}}^{-1})\,dx=\int_{\widetilde{U}_{i_{0}}}f\circ\varphi_{i_{0}}^{-1}\cdot v\circ\varphi_{i_{0}}^{-1}\,d\mu\circ\varphi_{i_{0}}^{-1}.
\end{equation}
Since $u$ is $\mu$-subharmonic, i.e., $f\geq0$ $\mu$-a.e., we have that
\begin{equation}\label{geq0}
f\circ\varphi_{i_{0}}^{-1}\geq0,\quad(\mu\circ\varphi_{i_{0}}^{-1})\text{-a.e.}.
\end{equation}
Moreover, $u\circ\varphi_{i_{0}}^{-1}\in C(\overline{\widetilde{U}_{i_{0}}})$ since $u\in C(\overline{\Omega})$. Combining Proposition \ref{mpi} and Lemma \ref{Lemma1}, we see that $u\circ\varphi_{i_{0}}^{-1}$ is constant on $\widetilde{U}_{i_{0}}$. Hence, $u$ is constant on $U_{i_{0}}$. Since $\Omega$ is connected and $\Omega\subseteq\bigcup_{i=1}^{N}U_{i}$, there exists some integer $i_{1}\in\{1,\ldots,N\}$ and $i_{1}\neq i_{0}$ such that $U_{i_{0}}\cap U_{i_{1}}\neq\emptyset$. Let $p_{1}\in U_{i_{0}}\cap U_{i_{1}}$. Then $u(p_{1})=u(p_{0})$. By a same argument, we can proof that $u$ is a constant on $U_{i_{1}}\cap\Omega$. Now, we select $U_{i_{2}}$ such that $U_{i_{2}} \cap\bigcup_{j=i_{0}}^{i_{1}}(U_{j}\cap\Omega)\neq\emptyset$. A similar argument implies that $u$ is constant on $\bigcup_{j=i_{0}}^{i_{2}}(U_{j}\cap\Omega)$. Repeat this process until all of $U_{i}$, $i=1,\ldots,N$ have been selected. Hence, $u$ is constant on $\Omega\cap U_{i}$, $i=1,\ldots,N$, which completes the proof.
\end{proof}

\begin{rmk}\label{rmkmin}
Substituting $u$ by $-u$ in the proof of Theorem \ref{thmmax}, we see that if $u\in C(\overline{\Omega})$ is a $\mu$-superharmonic function, then $u$ cannot attain its minimum value in $\Omega$.
\end{rmk}

A function $u\in{\rm dom}(\Delta_{\mu})$ is called $\mu$-harmonic if $\Delta_{\mu}u=0$ $\mu$-a.e.. It is known that a compact Riemannian manifold $M$ with $\partial M=\emptyset$ is closed (see, e.g., \cite{Lee_2013}). Hence, a classical harmonic function on $M$ is only constant. For a $\mu$-harmonic function, we have the following analogous property.
\begin{prop}\label{PMempty}
Let $M$ be a compact smooth and connected Riemannian manifold with $\partial M=\emptyset$. Let $\mu$ be a positive finite Borel measure on $M$. Assume $\mu$ satisfies (MPI) on $M$. Then any $\mu$-harmonic function $u$ on $M$ is a constant.
\end{prop}
\begin{proof}
Since $u$ is $\mu$-harmonic, by \cite[Proposition 3.3]{Ngai-Ouyang_2024}, for all $v\in C_{c}^{\infty}(M)$, we have
$$
\int_{M}-\Delta_{\mu}u\cdot v\,d\mu=\int_{M}\langle\nabla u, \nabla v\rangle\,d\nu=-\int_{M}u\Delta v\,d\nu=0.
$$
Therefore, by Weyl's Lemma (see e.g., \cite[(3.11)]{Yau_1976} or \cite{Weyl_1940}), $u$ is smooth and harmonic on $M$. Since $M$ is closed, we have that $u$ is a constant.
\end{proof}

\begin{prop}\label{PMNotempty}
Let $d\geq1$, $M$ be a complete smooth Riemannian $d$-manifold and $\Omega\subseteq M$ be a bounded domain with $\partial\Omega\neq\emptyset$. Let $\mu$ be a positive finite Borel measure on $M$ with {\rm supp}$(\mu)\subseteq\overline{\Omega}$ and $\mu(\Omega)>0$. Assume $\mu$ satisfies (MPI) on $\Omega$. If $u\in{\rm dom}(\Delta_{\mu})$ is a $\mu$-harmonic function, then $u\equiv0$.
\end{prop}
\begin{proof}
The argument in the proof of Proposition \ref{PMempty} implies that $u$ is smooth and harmonic on $\Omega$. Since $u\in W_{0}^{1,2}(\Omega)$ and $\Delta u=0\in L^{2}(\Omega)$, by Green's formula (see \cite[Lemma 4.4]{Grigoryan_2009}), we have
$$
\int_{\Omega}|\nabla u|^{2}\,d\nu=\int_{\Omega}\langle\nabla u, \nabla u\rangle\,d\nu=-\int_{\Omega}u\Delta u\,d\nu=0
$$
Combining this and the Poincar\'e inequality (\ref{eq1}), we have $\int_{\Omega}|u|^{2}\,d\nu=0.$ Note that $u$ is smooth on $\Omega$. Therefore $u\equiv0$ on $\Omega$.
\end{proof}

\section{Nodal domain theorem}
Let $d\geq1$, $M$, $\Omega$ and $\mu$ be as in Theorem \ref{Courant2}. Then the eigenvalues of $-\Delta_{\mu}$ can be ordered as $0\leq\lambda_{1}\leq\lambda_{2}\leq\cdots.$ Moreover, there exists an orthonormal basis $\{\phi_{i}\}_{i=1}^{\infty}$ of $L^{2}(\Omega,\mu)$ consisting of eigenfunctions of $-\Delta_{\mu}$ \cite[Theorem 2.2]{Ngai-Ouyang_2024}. Based on this spectral property, we obtain the following proposition.

\begin{prop}\label{prop>0}
Assume the hypotheses of Theorem \ref{Courant2}. Let $\{\lambda_{i}\}$ be the eigenvalues of $-\Delta_{\mu}$ given by (\ref{eigen}) or (\ref{eigen2}).
\begin{enumerate}
\item[(a)]~If $\partial \Omega\neq\emptyset$, then the Dirichlet eigenvalues satisfy $0<\lambda_{1}\leq\lambda_{2}\leq\cdots$.

\item[(b)]~If $\partial \Omega=\emptyset$, then the eigenvalues satisfy $0=\lambda_{0}<\lambda_{1}\leq\lambda_{2}\leq\cdots$.
\end{enumerate}
\end{prop}
\begin{proof}
\noindent(a) Since $\partial \Omega\neq\emptyset$, by Proposition \ref{PMNotempty}, if $0$ is an eigenvalue of $-\Delta_{\mu}$, then the unique function satisfying $\Delta_{\mu}u=0$ is $u\equiv0$. Hence, $0$ is not an eigenvalue of $-\Delta_{\mu}$. Now, part (a) follows by using \cite[Theorem 2.2]{Ngai-Ouyang_2024}.

\noindent(b) Let $f=C\neq0$ be a constant function on $\Omega$. Obviously, $f\in W^{1,2}(\Omega)$. Since $\Omega=\overline{\Omega}$ is compact, it is complete by Hopf-Rinow theorem. Hence $f\in W_{0}^{1,2}(\Omega)$ (see, e.g., \cite[Theorem 2.7]{Hebey_1996}). Moreover, $f\in L^{2}(\Omega,\mu)$ and $\|f\|_{L^{2}(\Omega,\mu)}\neq0$. Therefore, $f\in{\rm dom}(\mathcal{E})$. Note that the zero function $0$ belong to $L^{2}(\Omega,\mu)$. Thus,
$$\int_{\Omega}\langle\nabla f,\nabla v\rangle\,d\nu=0=\int_{\Omega}0\cdot v\,d\mu\qquad\text{for~all}~v\in C_{c}^{\infty}(\Omega).$$
Hence, it follows from \cite[Proposition 3.3]{Ngai-Ouyang_2024} that $f\in{\rm dom}(\Delta_{\mu})$ and $\Delta_{\mu}f=0$. Thus, $\lambda_{0}=0$ is an eigenvalue of $-\Delta_{\mu}$. Part (a) now follows by using \cite[Theorem 2.2]{Ngai-Ouyang_2024}.
\end{proof}

We use Rayleigh quotient (see, e.g., \cite{Deng-Ngai_2021,Ngai-Zhang-Zhao_2024}) to study the eigenvalues of $-\Delta_{\mu}$.
\begin{defi}\label{RQ}
Assume the hypotheses of Proposition \ref{prop>0}. Let $\mathcal{E}(\cdot,\cdot)$ be defined as in (\ref{form}) on $L^{2}(\Omega,\mu)$. For any $u\in\rm{dom(\mathcal{E})}$, define the {\em Rayleigh quotient}
\begin{equation}
R_{\mu}(u):=\frac{\mathcal{E}(u,u)}{\quad(u,u)_{L^2(\Omega,\mu)}}
=\frac{\int_{\Omega}|\nabla u|^{2} \,d\nu}{\displaystyle\int_{\Omega}|u|^{2} \,d\mu}.
\end{equation}
\end{defi}

Note that if $\partial \Omega=\emptyset$, then it follows from the proof of Proposition \ref{prop>0}(a) that a constant $\phi_{0}=C\neq0$ is an eigenfunction of $-\Delta_{\mu}$ with eigenvalue 0. In order to study the positive eigenvalues and nonconstant eigenfunctions of $-\Delta_{\mu}$, we let

\begin{equation}\label{HH}
\mathcal{H}_{\mu}:=\begin{cases}\{u\in{\rm dom}(\mathcal{E}):\int_{\Omega}u\,d\mu=0\},\quad &\text{if}~\partial \Omega=\emptyset,\\
{\rm dom}(\mathcal{E}),&\text{if}~\partial \Omega\neq\emptyset
\end{cases}
\end{equation}
and
\begin{equation}\label{HHH}
\widetilde{\mathcal{H}}_{\mu}:=\begin{cases}\{u\in L^{2}(\Omega,\mu):\int_{\Omega}u\,d\mu=0\},\quad &\text{if}~\partial \Omega=\emptyset,\\
L^{2}(\Omega,\mu),&\text{if}~\partial \Omega\neq\emptyset.
\end{cases}
\end{equation}
According to \cite[Theorem 2.2]{Ngai-Ouyang_2024} and Proposition \ref{prop>0}, there exists an orthonormal basis $\{\phi_{i}\}_{i=1}^{\infty}$ of $\widetilde{\mathcal{H}}_{\mu}$ consisting of non-constant eigenfunctions of $-\Delta_{\mu}$ on $\Omega$. The following lemma is a generalization of \cite[Lemma 4.1]{Ngai-Zhang-Zhao_2024}; we omit the proof.

\begin{lem}\label{lemRu}
Let $\{\lambda_{n}\}_{n\geq1}$ be the set of positive eigenvalues of $-\Delta_{\mu}$ with corresponding eigenfunctions $\{u_{n}\}_{n\geq1}$. Then
$$
\lambda_{n}=\begin{cases}{\rm min}\big\{R_{\mu}(u)\,\big|\,u\in \mathcal{H}_{\mu}\big\},\quad&  n=1,\\
{\rm min}\big\{R_{\mu}(u)\,\big|\,u\in \mathcal{H}_{\mu},\,(u,u_{i})_{L^2(\Omega,\mu)}=0,\,i=1,\ldots,n-1\big\},& n\geq2.
\end{cases}$$
Moreover, for all integers $n\geq1$, $R_{\mu}(u_{n})=\lambda_{n}$.
\end{lem}

The proof of the following corollary is standard; we include it in Appendix A.
\begin{coro}\label{corollary1}
Use the hypotheses of Lemma \ref{lemRu}. Let $\partial\Omega\neq\emptyset$. If $u\in\mathcal{H}_{\mu}$ satisfies $R_{\mu}(u)=\lambda_{n}$ for some eigenvalue $\lambda_{n}$, then $u$ is a $\lambda_{n}$-eigenfunction.
\end{coro}

\begin{proof}[Proof of Theorem \ref{Courant2}]
\noindent(a) Using Theorem \ref{thmmax}, the fact that $\lambda_{1}>0$, and an argument similar to that of \cite[Theorem 1.1]{Ngai-Zhang-Zhao_2024}, one can show that $u_{1}(x)\neq0$ for all $x\in\Omega$; we omit the details. Without loss of generality, we assume $u_{1}(x)>0$. Note that $u_{n}(x)$ is orthogonal to $u_{1}(x)$, i.e., $\int_{\Omega}u_{n}(x)u_{1}(x) \,d\mu=0.$ Hence, for $n\geq 2$, by the continuity of $u_{n}(x)$, the number of the nodal domain of $u_{n}(x)$ is at least $2$. Let $\mathcal{Z}_{\mu}$ be defined as in (\ref{nodalset}). Assume $\mathcal{Z}_{\mu}(u_{n})$ divides $\Omega$ into $m$ nodal domains $\Omega_{1},\ldots,\Omega_{m}$, where each $\Omega_{i}$ is a nodal domain. Let
$$
w_{j}(x)=\begin{cases}
u_{n}(x),\,\quad\qquad x\in\Omega_{j},\\
0,\,\qquad \qquad\quad x\in M\backslash\Omega_{j}.
\end{cases}
$$
and
\begin{equation}\label{TB88}
w(x)=\sum_{j=1}^{m}c_{j}w_{j}(x),
\end{equation}
where $c_{1},\ldots,c_{m}$ are arbitrary constants. By the same calculation as in the proof of \cite[Theorem 1.1]{Ngai-Zhang-Zhao_2024}, we have
\begin{equation}\label{lamde}
R_{\mu}(w)=\lambda_{n}.
\end{equation}
We can choose the coefficients $\{c_{j}\}_{j=1}^{m}$ of $w(x)$ in (\ref{TB88}) so that
\begin{equation}\label{leqs}
(w,u_{i})_{L^2(\Omega,\mu)}=0,\,\qquad i=1,\ldots,m-1.
\end{equation}
The existence of $\{c_{j}\}_{j=1}^{m}$ follows by the properties of the solutions of the linear system of equations in (\ref{leqs}). Hence, for this choice of $\{c_{j}\}_{j=1}^{m}$, Lemma \ref{lemRu} implies that $\lambda_{m}\leq R_{\mu}(w)$. Combining this and (\ref{lamde}), we see that $\lambda_{m}\leq\lambda_{n}$, which implies $m\leq n$.

\noindent(b) Assume $\partial \Omega=\emptyset$. Since, $u_{n}\in\widetilde{\mathcal{H}}_{\mu}$ for all $n\geq1$, we have $\int_{\Omega}u_{n}\,d\mu=0$. Therefore, by the continuity of $u_{n}$, the number of nodal domains of $u_{n}$ is at least 2. Assume $u_{n}$ has $m$ nodal domains. Using the argument in the proof part (a), we let $w(x)$ be defined as in (\ref{TB88}). By (\ref{lamde}), $R_{\mu}(w)=\lambda_{n}$. For the same reason as in part (a), we can choose the coefficients $\{c_{j}\}_{j=1}^{m}$ of $w(x)$ in (\ref{TB88}) such that $(w,u_{0})_{L^2(\Omega,\mu)}=0$, where $u_{0}=C$ is a constant, and
$$
(w,u_{i})_{L^2(\Omega,\mu)}=0,\qquad i=1,\ldots,m-2.
$$
Therefore, for this choice of $\{c_{j}\}_{j=1}^{m}$, Lemma \ref{lemRu} implies that $\lambda_{m-1}\leq R_{\mu}(w)$. Hence, $\lambda_{m-1}\leq\lambda_{n}$, which implies $m\leq n+1$.
\end{proof}
The following corollaries follow from Theorem \ref{Courant2} and its proof. We omit the proofs.
\begin{coro}
Assume the same hypotheses of Theorem \ref{Courant2}.
\begin{enumerate}
\item[(a)]~If $\partial \Omega\neq\emptyset$, then $u_{1}$ has only $1$ nodal domain and $u_{2}$ has $2$ nodal domains.

\item[(b)]~If $\partial \Omega=\emptyset$, then $u_{1}$ has $2$ nodal domains.
\end{enumerate}
\end{coro}

\begin{coro}
Assume the same hypotheses of Theorem \ref{Courant2}.  Let $n\geq2$. Assume that the eigenvalue $\lambda_{n}$ has multiplicity $r\geq1$.
\begin{enumerate}
\item[(a)]~If $\partial \Omega\neq\emptyset$, then $u_{n}$ has at most $n+r-1$ nodal domains.

\item[(b)]~If $\partial \Omega=\emptyset$, then $u_{n}$ has at most $n+r$ nodal domains.
\end{enumerate}
\end{coro}

\section{Linear Dirichlet Problem and Green operator on bounded domains}

In this section, we study the inverse operator of $\Delta_{\mu}$. Let $d\geq2$ and $M$ be a smooth complete Riemannian $d$-manifold. Let $\Omega\subseteq M$ be a bounded domain. Let $\M(\Omega)$ be the vector space of finite Borel measures on $\Omega$. In order to define the Green operator, we first consider the following linear Dirichlet problem with density $\mu$:
\begin{equation}\label{DP2}
\begin{cases}
\Delta u=\mu,\quad &\text{on}~\Omega,\\
u=0,\quad&\text{in}~\partial\Omega,
\end{cases}
\end{equation}
where $\mu\in\M(\Omega)$. Dirichlet problem (\ref{DP2}) was considered in, for instance, \cite{Littman-Stampacchia-Weinberger_1963,Ponce_2006} to study the regular points for elliptic equations and potential theory. We follow the method in \cite[Definition 3.1]{Ponce_2006} of choosing the space of test functions to define a weak solution of (\ref{DP2}).
\begin{defi}
Let $\Omega\subseteq M$ be a bounded domain, and let $\mu\in\M(\Omega)$. We call $u:\Omega\to\mathbb{R}$ a {\em weak solution} of (\ref{DP2}) if $u\in L^{1}(\Omega)$ and
\begin{equation}\label{weak}
\int_{\Omega}u\Delta v\,d\nu=-\int_{\Omega}v\,d\mu\qquad \text{for~all} ~v\in C_{c}^{\infty}(\Omega).
\end{equation}
\end{defi}

Let $\big(W_{0}^{1,2}(\Omega)\big)'$ be the dual space of the Sobolev space $W_{0}^{1,2}(\Omega)$.
\begin{prop}\label{prop3}
Let $d\geq2$, M be a smooth complete Riemannian $d$-manifold and $\Omega\subseteq M$ be a bounded domain. If $\mu\in\big(W_{0}^{1,2}\big(\Omega)\big)'$, then there exists an unique weak solution $u\in W_{0}^{1,2}(\Omega)$ of the Dirichlet problem (\ref{DP2}).
\end{prop}
\begin{proof}
Since $\mu\in\big(W_{0}^{1,2}(\Omega)\big)'$, there exists a constant $A>0$ such that for any $\varphi\in W_{0}^{1,2}(\Omega)$
$$|\mu[\varphi]|:=\Big|\int_{\Omega} \varphi\,d\mu\Big|\leq A\|\varphi\|_{W_{0}^{1,2}(\Omega)}.$$
Note that for any $u,v\in W_{0}^{1,2}(\Omega)$, $\int_{\Omega}\langle\nabla u,\nabla v\rangle\,d\nu$ is an inner product in $W_{0}^{1,2}(\Omega)$. The Riesz representation theorem implies that there exists an unique $u\in W_{0}^{1,2}(\Omega)$ such that for every $\varphi\in W_{0}^{1,2}(\Omega)$,
\begin{equation}\label{eq14}
\int_{\Omega}\varphi\,d\mu=\int_{\Omega}\langle\nabla u,\nabla \varphi\rangle\,d\nu
\end{equation}
Note that (\ref{eq14}) holds for all $\varphi\in C_{c}^{\infty}(\Omega)\subseteq W_{0}^{1,2}(\Omega)$ and
\begin{equation}\label{eq14a}
\int_{\Omega}\varphi\,d\mu=\int_{\Omega}\langle\nabla u,\nabla \varphi\rangle\,d\nu=-\int_{\Omega}u\Delta \varphi\,d\nu,
\end{equation}
where the last equation follows by Green formula (see e.g., \cite[Lemma 4.4]{Grigoryan_2009}). Hence $u$ is a weak solution of Dirichlet problem (\ref{DP2}). Moreover, one can check that the uniqueness of $u$ is reserved by the  denseness of $C_{c}^{\infty}(\Omega)\subseteq W_{0}^{1,2}(\Omega)$.
\end{proof}

Let $\mathcal{X}$ be a metric space, $\mu$ be a positive finite Borel measure on $\mathcal{X}$ with compact support, and $s>0$. We recall that $\mu$ is {\em upper $s$-regular} if there exists some constant $c>0$ such that
$$\mu(B^{\mathcal{X}}_{r}(x))\leq cr^{s}$$
for all $x\in{\rm supp}(\mu)$ and all $r\in[0,{\rm diam}({\rm supp}(\mu))]$ (see e.g., \cite{Hu-Lau-Ngai_2006}).

Note that the Bishop-Gromov inequality (see e.g., \cite{Bishop-Crittenden_1964,Anderson_1990}) implies that the Riemannian volume measure $\nu$ of $M$ is upper $d$-regular. It follows by \cite[Lemma 4.1]{Ngai-Ouyang_2024} that $\underline{\dim}_{\infty}(\nu)\geq d,$ where $\underline{\rm dim}_{\infty}(\nu)$ is defined as in (\ref{dim}). We have the following example; we omit the proof.
\begin{exam}
Let $d\geq2$, $M$ be a complete $d$-dimensional Riemannian manifold with nonnegative Ricci curvature and $\nu$ be the Riemannian volume measure. Then $\underline{\dim}_{\infty}(\nu)\geq d$.
\end{exam}

\begin{prop}\label{prop5}
Assume the hypotheses of Proposition \ref{prop3}. Let $\mu$ be a positive finite Borel measure on $\Omega$ with $\mu(\Omega)>0$ and supp$(\mu)\subseteq\overline{\Omega}$. Assume that $\underline{\dim}_{\infty}(\mu)>d-2$. Then for any $f\in L^{2}(\Omega,\mu)$, $f\mu\in\big(W_{0}^{1,2}(\Omega)\big)'$.
\end{prop}
\begin{proof}
Let $f\in L^{2}(\Omega,\mu)$. For any $\phi\in W_{0}^{1,2}(\Omega)$, under the assumption of $\underline{\dim}_{\infty}(\mu)>d-2$, $\phi$ has a unique $L^{2}(\Omega,\mu)$-representative $\overline{\phi}$. Hence, We have
\begin{align*}
|f\mu[\phi]|&:=\Big|\int_{\Omega}\overline{\phi}\cdot f\,d\mu\Big|\leq\|f\|_{L^{2}(\Omega,\mu)}\|\overline{\phi}\|_{L^{2}(\Omega,\mu)}\\
&\leq C \|f\|_{L^{2}(\Omega,\mu)}\|\nabla\phi\|_{L^{2}(\Omega)}\leq C \|f\|_{L^{2}(\Omega,\mu)}\|\phi\|_{W_{0}^{1,2}(\Omega)}.
\end{align*}
Therefore, for any $f\in L^{2}(\Omega,\mu)$, $f\mu\in\big(W_{0}^{1,2}(\Omega)\big)'$.
\end{proof}

Let $\Omega\subseteq M$ be a bounded domain on which the Green function $G_{y}(x)$ of $\Delta$ exists. It is known that $G_{y}(x)$ satisfies the equation (in distributional sense):
\begin{equation}\label{GD}
-\Delta G_{y}(x)=\delta_{y}, \quad \text{for~all}~y\in\Omega,
\end{equation}
where $\delta_{y}$ is the Dirac measure at $y$. For more properties of $G_{y}(x)$, see e.g., \cite{Grigoryan_1999,Li-Tam_1987,Grigoryan_2009} and the references therein. We summarize the properties of $G_{y}(x)$ below.
\begin{prop}\label{pg}
Let $d\geq2$, M be a smooth complete Riemannian $d$-manifold and $\Omega\subseteq M$ be a bounded domain on which the Green function $G_{y}(x)$ exists. Then $G_{y}(x)$ satisfies
\begin{enumerate}
\item[(a)] The on-diagonal value $G_{y}(y)$ is infinite and the singularity of $G_{y}(x)$ as $x\to y$ is of the same order as that in $\R^{d}$; that is
$$\widetilde{C}_{1}\log{\rho(x,y)^{-1}}\leq G_{y}(x)\leq\widetilde{C}_{2}\log{\rho(x,y)^{-1}},~\text{when}~d=2,$$
and
$$\widetilde{C}_{1}\rho(x,y)^{2-d}\leq G_{y}(x)\leq\widetilde{C}_{2}\rho(x,y)^{2-d},~\text{when}~d\geq3,$$
as $\rho(x,y)\to0$, where the constant $\widetilde{C}_{1}$ and $\widetilde{C}_{2}$ only depended on $d$.

\item[(b)] $G_{y}(x)>0$ and $G_{y}(x)=G_{x}(y)$ for any $x,y\in\Omega$.

\item[(c)] $G_{y}(x)$ is harmonic away from $y$.

\item[(d)] $G_{y}(x)=0$ for any $x\in\Omega$ and $y\in\partial\Omega$.
\end{enumerate}
Moreover, if $\partial\Omega$ is smooth, $G_{y}(x)$ is continuous up to the boundary $\partial\Omega$.
\end{prop}

Let $\mu$ be a positive finite Borel measure with supp$(\mu)\subseteq\overline{\Omega}$ and $\mu(\Omega)>0$. We need the following condition: there exists a constant $\overline{C}>0$ such that
\begin{equation}\label{C2}
\sup_{y\in\Omega}\int_{\Omega}G_{y}(x)\,d\mu(x)\leq \overline{C}<\infty.
\end{equation}
The following proposition generalizes \cite[Proposition 4.1]{Hu-Lau-Ngai_2006} to $\Omega\subseteq M$. We will omit some details in the proof.
\begin{prop}\label{prop4}
Assume the hypotheses of Theorem \ref{thcon}. Then condition (\ref{C2}) holds.
\end{prop}
\begin{proof}
Assume that $\underline{\dim}_{\infty}(\mu)>d-2$. By \cite[Lemma 4.1]{Ngai-Ouyang_2024}, $\mu$ is $\alpha$-regular for some $\alpha>d-2$, i.e., there exists a constant $C>0$ such that for all $x\in{\rm supp}(\mu)$ and all $r>0$,
\begin{equation}\label{eqregular}
\mu(B_{r}(x))<Cr^{\alpha}.
\end{equation}
The proposition follows by using properties of $G_{y}(x)$ and applying a similar argument as that in the proof of \cite[Proposition 4.1]{Hu-Lau-Ngai_2006}.
\end{proof}

Under the assumption of (\ref{C2}), we define the {\em Green's operator} $G_{\mu}$ on $L^{p}(\Omega,\mu)~(1\leq p\leq\infty)$ by
\begin{equation}\label{go}
(G_{\mu}f)(x)=\int_{\Omega}G_{y}(x)f(y)\,d\mu(y)
\end{equation}
By (\ref{C2}) and a similar calculation as in \cite[Proposition 4.2]{Hu-Lau-Ngai_2006}, we have that $G_{\mu}$ is bounded on $L^{p}(\Omega,\mu)$.
\begin{prop}\label{propLp}
Assume the hypotheses of Theorem \ref{thcon}. For any $f\in L^{p}(\Omega,\mu)$ with $1\leq p\leq\infty$, Let $G_{\mu}f$ be defined as (\ref{go}). Then there exists some constant $\widehat{C}>0$ such
\begin{equation}\label{eqnu}
\|G_{\mu}f\|_{L^{p}(\Omega)}\leq\widehat{C}\|f\|_{L^{p}(\Omega,\mu)}.
\end{equation}
\end{prop}
\begin{proof}
By Proposition \ref{prop4}, condition (\ref{C2}) holds for $\nu$.  we have, for $1<p<\infty$,
\begin{align*}
\int_{\Omega}\big|G_{\mu}f(x)\big|^{p}\,d\nu&=\int_{\Omega}\Big|\int_{\Omega}G_{y}(x)f(y)\,d\mu(y)\Big|^{p}\,d\nu(x)\\
&\leq\int_{\Omega}\Big(\int_{\Omega}G_{y}(x)|f(y)|^{p}\,d\mu(y)\Big)\Big(\int_{\Omega}G_{y}(x)\,d\mu(y)\Big)^{p-1}\,d\nu(x)\\
&\leq\overline{C}_{1}^{p-1}\int_{\Omega}\Big(\int_{\Omega}G_{y}(x)\,d\mu(y)\Big)|f(y)|^{p}\,d\mu(y)\\
&\leq \overline{C}_{1}^{p-1}C\|f\|_{L^{p}(\Omega,\mu)}^{p}.
\end{align*}
Thus (\ref{eqnu}) holds for $p\in(0,\infty)$. The cases $p=1$ and $p=\infty$ are clear.
\end{proof}

\begin{prop}\label{prop6}
Assume the hypotheses of Theorem \ref{thcon}. Then for any $f\in L^{2}(\Omega,\mu)$, $G_{\mu}f$ is a weak solution of Dirichlet problem:
\begin{equation}\label{D2}
\begin{cases}
\Delta u=f\mu,~&\text{in}~\Omega,\\
u=0,~&\text{on}~\partial\Omega.
\end{cases}
\end{equation}
\end{prop}
\begin{proof}
Combining Propositions \ref{prop3} and \ref{prop5}, we see that the Dirichlet problem (\ref{D2}) has a unique $L^{2}$ weak solution. It follows by Proposition \ref{propLp} that $G_{\mu}f\in L^{2}(\Omega)$. Moreover,
for any $\xi\in C_{c}^{\infty}(\Omega)$,
\begin{align}\label{Ge}
\int_{\Omega}\Big(\int_{\Omega}G_{y}(x)f(y)\,d\mu(y)\Big)\Delta\xi\,d\nu&=\int_{\Omega}\Big(\int_{\Omega}G_{y}(x)\,d(f\mu)(y)\Big)\Delta\xi\,d\nu\notag\\
&=\int_{\Omega}\int_{\Omega}G_{y}(x)\Delta\xi(x)\,d\nu(x)d(f\mu)(y)\qquad\qquad(\text{Fubini})\notag\\
&=\int_{\Omega}\xi(y)\,d(f\mu)(y).
\end{align}
Therefore, $\int_{\Omega}G_{y}(x)f(y)\,d\mu(y)$ is a weak solution of (\ref{D2}).
\end{proof}

\begin{prop}\label{prop++1}
Assume the hypotheses of Proposition \ref{prop6}. Then there exists a smooth and harmonic function $h_{\mu}(f)\in L^{2}(\Omega,\mu)$ depending on $f$ and $\mu$ such that $(G_{\mu}f+h_{\mu}(f))\in W_{0}^{1,2}(\Omega)$.
\end{prop}
\begin{proof}
Since $\underline{\dim}_{\infty}(\mu)>d-2$ and $f\in L^{2}(\Omega,\mu)$, it follows from Proposition \ref{prop5}, we have $f\mu\in\big(W_{0}^{1,2}(\Omega)\big)'$. In view of Proposition \ref{prop3}, there exists a unique weak solution $u$ in $W_{0}^{1,2}(\Omega)$. Note that by Proposition \ref{prop6}, $G_{\mu}f$ is a weak solution of (\ref{DP2}). Hence, by (\ref{weak}), we have for any $\xi\in C_{c}^{\infty}(\Omega)$,
$$\int_{\Omega}(u-G_{\mu}f)\Delta\xi\,d\nu=0,$$
which implies $u-G_{\mu}f$ is harmonic in the sense of distribution. Therefore, by \cite[Section 4.1]{Grigoryan_1999} or Weyl's lemma (see, e.g., \cite[(3.11)]{Yau_1976},\cite{Weyl_1940}), $u-G_{\mu}f$ is smooth and harmonic on $\Omega$. Let $h_{\mu}f:=u-G_{\mu}f$. Since $u\in W_{0}^{1,2}(\Omega)$, (\ref{MPI}) implies $u$ has a representative in $L^{2}(\Omega,\mu)$. Combining this and the fact that $G_{\mu}f\in L^{2}(\Omega,\mu)$, we have $h_{\mu}(f)\in L^{2}(\Omega,\mu)$. Moreover, $h_{\mu}f+G_{\mu}f=u\in W_{0}^{1,2}(\Omega)$.
\end{proof}

Since $u\in W_{0}^{1,2}(\Omega)\subseteq L^{2}(\Omega)$ and by Proposition \ref{propLp}, $G_{\mu}f\in L^{2}(\Omega)$, we have $h_{\mu}(f)\in L^{2}(\Omega)$. Hence, we have the following remark.
\begin{rmk}
Assume the hypotheses of Proposition \ref{prop6}. Then for any $f\in L^{2}(\Omega,\mu)$, $h_{\mu}(f)\in L^{2}(\Omega)$.
\end{rmk}
\begin{theo}\label{thGmu}
Assume the hypotheses of Proposition \ref{prop6}. Then for any $f\in L^{2}(\Omega,\mu)$,
$$-\Delta_{\mu}(G_{\mu}f+h_{\mu}(f))=f.$$
Moreover, $\big(G_{\mu}+h_{\mu}\big)\big(L^{2}(\Omega,\mu)\big)\subseteq{\rm dom}(-\Delta_{\mu})$ and $G_{\mu}+h_{\mu}=-\Delta_{\mu}^{-1}$.
\end{theo}
\begin{proof}
By Proposition \ref{prop++1}, we have that for any $f\in L^{2}(\Omega,\mu)$, $(G_{\mu}f+h_{\mu}f)\in W_{0}^{1,2}(\Omega)$. Using a similar argument as in the proof of \cite[Theorem 1.3]{Hu-Lau-Ngai_2006}, we have $(G_{\mu}f+h_{\mu}(f))\in{\rm dom}(\mathcal{E})$. Moreover, it follows by (\ref{Ge}) and the fact that $h_{\mu}(f)$ is harmonic, we have $-\Delta\big(G_{\mu}f+h_{\mu}(f)\big)=f\,d\mu$ in the sense of distribution. Hence, for any $f\in L^{2}(\Omega,\mu)$, it follows from \cite[Proposition 3.3]{Ngai-Ouyang_2024}, we have that $(G_{\mu}f+h_{\mu}f)\in{\rm dom}(-\Delta_{\mu})$ and $-\Delta_{\mu}\big(G_{\mu}f+h_{\mu}(f)\big)=f$. Consequently, $(G_{\mu}+h_{\mu})(L^{2}(\Omega,\mu))\subseteq{\rm dom}(-\Delta_{\mu})$. In view of \cite[Theorem 3.4]{Ngai-Ouyang_2024}, we have $G_{\mu}+h_{\mu}=-\Delta_{\mu}^{-1}$.
\end{proof}

\section{Continuity of Eigenfunctions on bounded domain}

Throughout this section, let $d$, $M$, $\Omega$, $\mu$ be as in Theorem \ref{thcon}. We will study the continuity of eigenfunctions of $-\Delta_{\mu}$ and prove Theorem \ref{thcon}.

Let $G_{y}(x)$ be the Green function defined in (\ref{GD}). In this section, we will make frequent use of the properties of Green function in Proposition \ref{pg}. For convenience, for integer $k\geq1$, we let
\begin{equation}\label{Qk}
Q_{k}(y):=\big\{x\in\Omega:\rho(x,y)\in[2^{-k},2^{-(k-1)}]\big\},
\end{equation}
and let
\begin{equation}\label{Nr}
V_{r}(y):=B^{M}_{r}(y)\cap\Omega=\{x\in\Omega:\rho(x,y)<r\}.
\end{equation}
For any $y\in\Omega$, Proposition \ref{pg}(c) implies that the Green function is harmonic on $\Omega\setminus\{y\}$. Since $G_{y}(x)$ is continuous up to $\partial\Omega$ and vanishes on $\partial\Omega$, it follows by the maximum principle (see, e.g., \cite[Section XII-11]{Chavel_1999}) that
\begin{equation}\label{decreasing}
\sup_{\Omega\setminus V_{r}(y)}G_{y}(x)=\max_{x\in\partial V_{r}(y)}G_{y}(x)
\end{equation}
for $r>0$ such that $\Omega\setminus V_{r}(y)\neq\emptyset$.
\begin{lem}\label{lem2}
Let $d=2$ and assume that the hypotheses of Theorem \ref{thGmu} are satisfied. Let $G_{y}(\cdot)$ be the Dirichlet Green function on $\Omega$. Then $G_{y}(\cdot)\in L^{2}(\Omega,\mu)$ for each $y\in\Omega$.
\end{lem}
\begin{proof}
Let $y\in\Omega$ be arbitrary and $V_{r}(y)$ be defined as in (\ref{Nr}). By Proposition \ref{pg}(a), there exists a constant $\widetilde{C}_{1}>0$ such that for some $r_{1}\in(0,1)$ sufficiently small and all $x\in V_{r_{1}}(y)$,
\begin{equation}\label{eq15}
G_{y}(x)\leq \widetilde{C}_{1}|\log{\rho(x,y)}|.
\end{equation}
By (\ref{decreasing}), $G_{y}(x)|_{\Omega\setminus V_{r_{1}}(y)}\leq\displaystyle\max_{x\in \partial V_{r_{1}}(y)}G_{y}(x)\leq\widetilde{C}_{1}|\log{r_{1}}|.$
Therefore, to prove the lemma, it suffices to prove that $\int_{V_{r_{1}}(y)}|G_{y}(x)|^{2}\,d\mu$ is finite. Let $Q_{k}(y)$ be defined as in (\ref{Qk}). By (\ref{eq15}) and the fact that $V_{r_{1}(y)}\subseteq V_{1}(y)$, we have
\begin{align}\label{eq15a}
\int_{V_{r_{1}}(y)}|G_{y}(x)|^{2}\,d\mu&\leq\widetilde{C}_{1}^{2}\int_{V_{r_{1}}(y)}|\log{\rho(x,y)}|^{2}\,d\mu=\widetilde{C}_{1}^{2}\sum_{k=1}^{\infty}\int_{Q_{k}(y)}|\log{\rho(x,y)}|^{2}\,d\mu\notag\\
&\leq\widetilde{C}_{1}^{2}\sum_{k=1}^{\infty}(\log{2^{k}})^{2}\mu(V_{2^{-(k-1)}}(y))\notag\\
&\leq\widetilde{C}_{1}^{2}(\log2)^{2}\sum_{k=1}^{\infty}k^{2}2^{-\alpha(k-1)}<\infty,\qquad\qquad(\text{by}~(\ref{eqregular}))
\end{align}
where $\alpha>0$ is a constant.
\end{proof}

\begin{rmk}\label{rmk1}
Assume the hypotheses of Lemma \ref{lem2}. Then there exists a constant $\overline{C}_{2}>0$ such that for all $y\in\Omega$, $\|G_{y}(\cdot)\|_{L^{2}(\Omega,\mu)}\leq\overline{C}_{2}$.
\end{rmk}
\begin{proof}
We use the notations and proof of Lemma \ref{lem2}. By (\ref{decreasing}), for any fixed $y\in\Omega$ and all $x\in\Omega\setminus V_{r_{1}}(y)$, we have $G_{y}(x)\leq \max_{z\in\partial V_{r_{1}}(y)}G_{y}(z)\leq\widetilde{C}_{1}|\log{r_{1}}|.$
 Note that this bound is independent of the choice of $y$. Combining this and (\ref{eq15a}), we have
$$\|G_{y}(\cdot)\|_{L^{2}(\Omega,\mu)}\leq\left(\widetilde{C}_{1}^{2}(\log2)^{2}\sum_{k=1}^{\infty}k^{2}2^{-\alpha(k-1)}+\mu(\Omega)\cdot\widetilde{C}_{1}^{2}(\log{r_{1}})^{2}\right)^{\frac{1}{2}}=:\overline{C}_{2},$$
where $\overline{C}_{2}$ is independent of $y$.
\end{proof}

\begin{lem}\label{lem3}
Let $d\geq3$ and assume the hypotheses of Theorem \ref{thGmu}. Let $G_{y}(x)$ be the Dirichlet Green function on $\Omega$ and let $f\in{\rm dom}(-\Delta_{\mu})$. Then $G_{\mu}f^{2}$ is bounded.
\end{lem}

\begin{proof}
Since $f\in{\rm Dom(-\Delta_{\mu})}\subseteq W_{0}^{1,2}(\Omega)$ and $\underline{\dim}_{\infty}(\mu)>d-2$, there exists $f_{n}\in C_{c}^{\infty}(\Omega)$ such that $f_{n}\to f$ in $L^{2}(\Omega,\mu)$ as $n\to\infty$. We claim that there exists some constant $\widetilde{C}_{3}>0$ such that
\begin{equation}\label{eq16}
\lim_{n\to\infty}|G_{\mu}f^{2}-G_{\mu}f^{2}_{n}|\leq\widetilde{C}_{3}.
\end{equation}
By the definition of $G_{\mu}f$, we have, for all $x\in\Omega$,
$$|G_{\mu}f^{2}(x)-G_{\mu}f^{2}_{n}(x)|\leq\int_{\Omega}G_{y}(x)|f^{2}(y)-f_{n}^{2}(y)|\,d\mu(y).$$
Hence, to prove (\ref{eq16}), it suffices to prove that for all $x\in\Omega$,
$$\lim_{n\to\infty}\int_{\Omega}G_{y}(x)|f^{2}(y)-f_{n}^{2}(y)|\,d\mu(y)\leq\widetilde{C}_{3}.$$
Note that, by H\"older's inequality,
\begin{equation}\label{eq17}
\int_{\Omega}|f^{2}(y)-f_{n}^{2}(y)|\,d\mu(y)\leq\|f-f_{n}\|_{L^{2}(\Omega,\mu)}\cdot\|f+f_{n}\|_{L^{2}(\Omega,\mu)}\to0
\end{equation}
as $n\to\infty$. By Proposition \ref{pg}(a), there exists $r_{2}\in(0,1)$ sufficiently small such that
\begin{equation}\label{eq18}
G_{y}(x)\leq\widetilde{C}_{2}|\rho(x,y)|^{-(d-2)}
\end{equation}
when $\rho(x,y)\leq r_{2}$. Let $r_{0}:=\sup_{x,y\in\Omega}\rho(x,y)$. Then
\begin{align}\label{eq18a}
\int_{\Omega}G_{y}(x)|f^{2}(y)-f_{n}^{2}(y)|\,d\mu(y)\leq&\int_{\rho(x,y)<r_{2}}G_{y}(x)|f^{2}(y)-f_{n}^{2}(y)|\,d\mu(y)\notag\\
&+\int_{r_{2}\leq\rho(x,y)\leq r_{0}}G_{y}(x)|f^{2}(y)-f_{n}^{2}(y)|\,d\mu(y).
\end{align}
The second integral on the right-hand side tends to $0$ by using (\ref{eq17}) and the fact that $G_{y}(x)$ is bounded on $\rho(x,y)\in[r_{2},r_{0}]$. Moreover, it follows by (\ref{decreasing}) that
$$G_{y}(x)|_{\Omega\setminus V_{r_{2}}(x)}\leq\max_{y\in\partial V_{r_{2}}(x)}G_{y}(x)\leq\widetilde{C}_{2}\rho(x,y)^{-(d-2)}|_{\partial V_{r_{2}}(x)}=\widetilde{C}_{2}r_{2}^{-(d-2)}.$$
Note that this bound is independent of $x$. Hence $G_{y}(x)$ is uniformly bounded on $\rho(x,y)\in[r_{2},r_{0}]$. Let $Q_{k}:=Q_{k}(x)$. The first integral in (\ref{eq18a}) can be estimated as follows:
\begin{align}\label{eq19}
\int_{\rho(x,y)<r_{2}}G_{y}(x)|f^{2}(y)-f_{n}^{2}(y)|\,d\mu(y)&\leq\widetilde{C}_{2}\int_{\rho(x,y)<r_{2}}\rho(x,y)^{-(d-2)}|f^{2}(y)-f_{n}^{2}(y)|\,d\mu(y)\notag\\
&\leq\widetilde{C}_{2}\sum_{k=1}^{\infty}\int_{Q_{k}}\rho(x,y)^{-(d-2)}|f^{2}(y)-f_{n}^{2}(y)|\,d\mu(y)\notag\\
&\leq\widetilde{C}_{2}\sum_{k=1}^{\infty}2^{k(d-2)}\int_{Q_{k}}|f^{2}(y)-f_{n}^{2}(y)|\,d\mu(y)
\end{align}
Let $N\in\mathbb{N}$ be arbitrary. Note that by (\ref{eq17}), for each $k\in\{1,\ldots,N\}$, there exists $n_{N}(k)\in\mathbb{N}^{+}$, independent of $x$, such that
$$\int_{Q_{k}}|f^{2}(y)-f^{2}_{n_{N}}(y)|\,d\mu\leq\int_{\Omega}|f^{2}(y)-f^{2}_{n_{N}}(y)|\,d\mu\leq2^{-k(d-1)}.$$
Hence,
$$\sum_{k=1}^{N}2^{k(d-2)}\int_{Q_{k}}|f^{2}(y)-f^{2}_{n_{N}(k)}(y)|\,d\mu\leq\sum_{k=1}^{N}2^{-k}.$$
Letting $N\to\infty$, we have

\begin{equation}\label{eq20}
\sum_{k=1}^{\infty}2^{k(d-2)}\int_{Q_{k}}|f^{2}(y)-f^{2}_{n_{N}(k)}(y)|\,d\mu(y)\leq\sum_{k=1}^{\infty}2^{-k}=1.
\end{equation}
Combining this with (\ref{eq19}) proves (\ref{eq16}). Note that (\ref{eq16}) implies that there exists an integer $N_{0}$ sufficiently large such that for all $n>N_{0}$ and all $x\in\Omega$,
$$|G_{\mu}\big(f^{2}(x)-f^{2}_{n}(x)\big)|\leq \widetilde{C}_{4},$$
where $\widetilde{C}_{4}$ is independent of $x$. In particular,
\begin{align}\label{eq22}
|G_{\mu}f^{2}(x)|&\leq|G_{\mu}f^{2}_{N_{0}+1}(x)|+\widetilde{C}_{4}\leq\|f^{2}_{N_{0}+1}\|_{L^{\infty}(\Omega)}\int_{\Omega}G_{y}(x)\,d\mu(y)+\widetilde{C}_{4}\notag\\
&\leq\overline{C}_{1}\|f^{2}_{N_{0}+1}\|_{L^{\infty}(\Omega)}+\widetilde{C}_{4}=:\widetilde{C}_{5},
\end{align}
which completes the proof.
\end{proof}

\begin{prop}\label{bounded}
Assume the hypotheses of Theorem \ref{thcon}. Let $f\in{\rm dom}(-\Delta_{\mu})$. Then $G_{\mu}f$ is bounded.
\end{prop}
\begin{proof}
We divide the proof into two cases.

\noindent{\em Case 1. $d=2$.} By Lemma \ref{lem2}, for any $y\in\Omega$, $G_{y}\in L^{2}(\Omega,\mu)$. Hence,
\begin{align*}
|G_{\mu}f(x)|&=\Big|\int_{\Omega}G_{y}(x)f(y)\,d\mu(y)\Big|\leq\|G_{x}(\cdot)\|_{L^{2}(\Omega,\mu)}\|f\|_{L^{2}(\Omega,\mu)}\\
&\leq\overline{C}_{2}\|f\|_{L^{2}(\Omega,\mu)}.\qquad\qquad\qquad{(\text{by~Remark~\ref{rmk1}}})
\end{align*}

\noindent{\em Case 2. $d\geq3$.} By Lemma \ref{lem3}, there exists some constant $\widetilde{C}_{5}>0$ such that for any $f\in{\rm dom}(-\Delta_{\mu})$ and all $x\in\Omega$,
\begin{equation}\label{eq22}
|G_{\mu}f^{2}(x)|\leq\widetilde{C}_{5}.
\end{equation}
Hence,
\begin{align*}
|G_{\mu}f(x)|^{2}&=\Big|\int_{\Omega}G_{y}(x)f(y)\,d\mu(y)\Big|^{2}\leq\Big(\int_{\Omega}|G_{y}^{1/2}(x)f(y)\cdot G_{y}^{1/2}(x)|\,d\mu(y)\Big)^{2}\\
&\leq\int_{\Omega}G_{y}(x)f^{2}(y)\,d\mu(y)\cdot\int_{\Omega}G_{y}(x)\,d\mu(y)\\
&\leq\widetilde{C}_{5}\cdot\overline{C}.\qquad\qquad\qquad\qquad\qquad(\text{by~(\ref{C2})})
\end{align*}
Combining Cases 1 and 2 completes the proof.
\end{proof}

The following proposition is obvious; we omit the proof.
\begin{prop}\label{propnoatomic}
Let $\mathfrak{X}$ be a complete metric space and $\Omega\subseteq\mathfrak{X}$ be a precompact subset. Let $\mu$ be a positive finite regular Borel measure defined on $\mathfrak{X}$ with compact support $supp(\mu)\subseteq\overline{\Omega}$ and $\mu(\Omega)>0$. Let $\underline{{\rm dim}}_{\infty}(\mu)$ be defined as in (\ref{dim}). If $\underline{{\rm dim}}_{\infty}(\mu)>0$, then $\mu$ has no atomic.
\end{prop}

\begin{prop}\label{continuous}
Assume the hypotheses of Theorem \ref{thcon} are satisfied. Let $f\in{\rm dom}(-\Delta_{\mu})$. Then $G_{\mu}f$ is continuous on $\Omega$.
\end{prop}
\begin{proof}
We divide the proof into three steps.

\noindent{\em Step 1.} We claim that for any $\epsilon>0$, there exists constant $\widetilde{r}>0$ and $f_{1}\in L^{2}(\Omega,\mu)$ such that $|G_{\mu}f_{1}|<\epsilon$.
In the case $d=2$, by Lemma \ref{lem2}, we see that for each $y\in\Omega$, $G_{y}(\cdot)\in L^{2}(\Omega,\mu)$. Moreover, by Remark \ref{rmk1}, there exists $\overline{C}_{2}$ such that for all $y\in\Omega$, $\|G_{y}(\cdot)\|_{L^{2}(\Omega,\mu)}\leq\overline{C}_{2}$. Since $\underline{\rm dim}_{\infty}(\mu)>d-2=0$, it follows by Proposition \ref{propnoatomic} that $\mu$ is a continuous measure. Therefore,
$$\lim_{r\to0^{+}}\int_{B_{r}(z)}|G_{y}(x)|^{2}\,d\mu(y)=0\qquad \text{for~all}~x\in\Omega.$$
Hence,
for any $\epsilon>0$, there exists $\widetilde{r}_{1}>0$ sufficiently small such that for all $x\in\Omega$,
\begin{equation}\label{eq+1}
\int_{B_{\widetilde{r}_{1}}(z)}|G_{y}(x)|^{2}\,d\mu(y)\leq\epsilon^{2}.
\end{equation}
Now consider the case $d\geq3$. Let $\epsilon>0$. Then by Proposition \ref{prop4}, (\ref{C2}), and the continuity of $\mu$ again, there exists $\widetilde{r}_{2}>0$ sufficiently small such that  for all $x\in\Omega$,
\begin{equation}\label{eq+2}
\int_{B_{\widetilde{r}_{2}}(z)}G_{y}(x)\,d\mu(y)<\epsilon^{2}.
\end{equation}
Let $\widetilde{r}:=\min\{\widetilde{r}_{1},\widetilde{r}_{2}\}$ and $f_{1}:=f\chi_{B_{\widetilde{r}}(z)}$. For the case $d=2$, by H\"{o}lder's inequality and (\ref{eq+1}), we have
\begin{align}\label{eq5.15}
|G_{\mu}f_{1}|&=\Big|\int_{\Omega}G_{y}(x)f_{1}(y)\,d\mu(y)\Big|\leq\|f_{1}\|_{L^{2}(\Omega,\mu)}\Big(\int_{B_{\widetilde{r}}(z)}G_{y}^{2}(x)\,d\mu(y)\Big)^{1/2}\notag\\
&\leq\epsilon\|f\|_{L^{2}(\Omega,\mu)}.
\end{align}
For the case $d\geq3$, by H\"{o}lder's inequality and (\ref{eq+2}), we have
\begin{align*}
\big|G_{\mu}f_{1}\big|^{2}&=\left|\int_{\Omega}G_{y}(x)f_{1}(y)\,d\mu(y)\right|^{2}\leq\int_{B_{\widetilde{r}}(z)}G_{y}(x)f_{1}^{2}(y)\,d\mu(y)\cdot\int_{B_{\widetilde{r}}(z)}G_{y}(x)\,d\mu(x)\\
&\leq\epsilon^{2}\int_{\Omega}G_{y}(x)f^{2}(y)\,d\mu(y).
\end{align*}
Combining this and (\ref{eq22}), we have
\begin{equation}\label{eq(5.33)}
\big|G_{\mu}f_{1}\big|\leq\widetilde{C}_{5}\cdot\epsilon.
\end{equation}
Therefore, combining (\ref{eq5.15}) and (\ref{eq(5.33)}) proves the claim.

{\em Step 2.} Let $f_{2}:=f-f_{1}$. We claim that $G_{\mu}f_{2}$ is continuous at $z$. In fact, by Proposition \ref{pg}(a), for any $y\in\Omega\setminus B_{\widetilde{r}}(z)$, $G_{y}(z)$ is continuous at $z$, i.e.,
\begin{equation}\label{eq5.15aa}
\lim_{x\to z}G_{y}(x)=G_{y}(z).
\end{equation}
Note that there exists $\delta\in(0,\widetilde{r}/2)$ such that for any $x\in B_{\delta}(z)$ and all $y\in\Omega\setminus B_{\widetilde{r}}(z)$
$$|G_{y}(x)|\leq\widetilde{C}_{6}:=\max\left\{\widetilde{C}_{1}\Big|\log\frac{\widetilde{r}}{2}\Big|,\widetilde{C}_{2}\left(\frac{\widetilde{r}}{2}\right)^{2-d}\right\},$$
which implies that for all $y\in\Omega\setminus B_{\widetilde{r}}(z)$ and $x\in B_{\delta}(z)$
\begin{equation}\label{equbed}
|G_{y}(x)f_{2}(y)|\leq \widetilde{C}_{6}|f_{2}(y)|.
\end{equation}
Moreover, $f_{2}\in L^{1}(\Omega,\mu)$ as $\|f_{2}\|_{L^{1}(\Omega,\mu)}\leq\|f\|_{L^{1}(\Omega,\mu)}$. Therefore, by (\ref{eq5.15aa}), (\ref{equbed}) and dominated convergence theorem, we have
\begin{align*}
\lim_{x\to z}G_{\mu}f_{2}(x)&=\lim_{x\to z}\int_{\Omega}G_{y}(x)f_{2}(y)\,d\mu(y)=\lim_{x\to z}\int_{\Omega\setminus B_{\widetilde{r}}(z)}G_{y}(x)f_{2}(y)\,d\mu(y)\\
&=\int_{\Omega\setminus B_{\widetilde{r}}(z)}G_{y}(z)f_{2}(y)\,d\mu(y)=G_{\mu}f_{2}(z).
\end{align*}
{\em Step 3.} By step 2, For any $\epsilon>0$, there exists $\delta>0$ such that for any $\widetilde{z}\in B_{\delta}(z)$, $\big|G_{\mu}f_{2}(z)-G_{\mu}f_{2}(\widetilde{z})\big|<\epsilon$. Combining this, step 1 and the definition of $f_{1}$ and $f_{2}$, we have
\begin{align*}
\left|G_{\mu}f(z)-G_{\mu}f(\widetilde{z})\right|&=\left|G_{\mu}f_{1}(z)+G_{\mu}f_{2}(z)-G_{\mu}f_{1}(\widetilde{z})-G_{\mu}f_{2}(\widetilde{z})\right|\\
&\leq\left|G_{\mu}f_{1}(z)\right|+\left|G_{\mu}f_{1}(\widetilde{z})\right|+\left|G_{\mu}f_{2}(z)-G_{\mu}f_{2}(\widetilde{z})\right|\leq3\epsilon,
\end{align*}
which shows that $G_{\mu}f$ is continuous at $z$. The arbitrary of $z$ implies that $G_{\mu}f\in C(\Omega)$.
\end{proof}

\begin{proof}[Proof of Theorem \ref{thcon}]
we claim that $u\in{\rm dom}(-\Delta_{\mu})$ is a $\lambda$-eigenfunction of $-\Delta_{\mu}$ if and only if $u=\lambda(G_{\mu}u+h_{\mu}u)$. On the one hand, if $u\in{\rm dom}(\Delta_{\mu})$ satisfies $u=\lambda(G_{\mu}u+h_{\mu}(u))$, then by Theorem \ref{thGmu}, we have $-\Delta_{\mu}u=-\lambda\Delta_{\mu}(G_{\mu}u+h_{\mu}(u))=\lambda u,$ which implies $u$ is a $\lambda$-eigenfunction of $-\Delta_{\mu}$. On the other hand, if $-\Delta_{\mu} u=\lambda u$ for some $\lambda>0$, then  by Theorem \ref{thGmu} again, we have
$$u=\lambda(-\Delta_{\mu})^{-1}u=\lambda(G_{\mu}u+h_{\mu}(u)).$$
Therefore, to prove the continuity of the eigenfunction $u$ on $\Omega$, it suffices to prove that $G_{\mu}u+h_{\mu}(u)$ is continuous, but this follows by combining Proposition \ref{continuous} and the fact that $h_{\mu}(u)$ is smooth on $\Omega$.
\end{proof}

\section{Continuous eigenfunctions on compact Riemannian manifolds}

In this section, we let $M$ be a compact connected smooth Riemannian manifold with $\partial M=\emptyset$ and $\mu$ be a positive finite Borel measure measure defined on $M$. We study the continuity of eigenfunctions of $-\Delta_{\mu}$ defined on $M$ and prove Theorem \ref{CCT}.

Let $\mathcal{M}(M)$ be the vector space of finite Borel measures on $M$. We consider the following equation with density $\mu\in\mathcal{M}(M)$:
\begin{equation}\label{2.1}
\Delta u=\mu.
\end{equation}
\begin{defi}
Let $M$ be a compact connected smooth Riemannian manifold with $\partial M=\emptyset$ and $\mu\in\mathcal{M}(M)$. We call $u: M\to\mathbb{R}$ a {\em weak solution} of (\ref{2.1}) if $u\in L^{1}(M)$ and
\begin{equation}\label{7.2}
\int_{M}u\Delta v\,d\nu=-\int_{M}v\,d\mu
\end{equation}
for all $v\in C_{c}^{\infty}(M)$ satisfying $\int_{M}v\,d\nu=0$.
\end{defi}

Let $C>0$ be a constant and {\rm span}$\{C\}$ be the subspace of $W^{1,2}(M)$ spanned by $C$. Let
\begin{equation}\label{WM}
\mathcal{W}(M):={\rm span}\{C\}^{\bot}=\left\{f\in W^{1,2}(M):\int_{M}f\,d\nu=0\right\}.
\end{equation}
Obviously, $\mathcal{W}(M)$ is a closed subspace of $W^{1,2}(M)$. Hence $\mathcal{W}(M)$ is a Hilbert space (see e.g., \cite[Theorem 3.2--4]{Kreyszig_1978}).

Let $M$ be a compact smooth Riemannian manifold. The Green function on $M$ satisfies
\begin{equation}\label{GOC}
\Delta G_{y}(x)=\delta_{x}(y)-\nu(M)^{-1}
\end{equation}
in the sense of distribution. We state the following properties of $G_{y}(x)$ without proof; details can be found in \cite[Theorem 4.13]{Aubin_1982}.

\begin{thm}\label{pg2}
Let $M$ be a compact connected smooth Riemannian manifold. There exists $G_{y}(x)$, a Green's function of the Laplacian which has the following properties:
\begin{enumerate}
\item[(a)]~For all functions $f\in C^{2}(M)$:
\begin{equation}\label{deltaun}
\int_{M}G_{y}(x)\Delta f(x)\,d\nu(x)=f(y)-\nu(M)^{-1}\int_{M}f(x)\,d\nu(x).
\end{equation}
\item[(b)]~ $G_{y}(x)$ is smooth on $(M\times M)\setminus\{(x,x): x\in M\}$.
\item[(c)]~There exists a constant $\widehat{C}$ such that
\begin{equation}
|G_{y}(x)|<
\begin{cases}
\widehat{C}(1+|\log{\rho(x,y)}|),&\text{for}~d=2,\\
\widehat{C}\rho(x,y)^{2-d},&\text{for}~d\geq3,
\end{cases}
\end{equation}
where $\rho(x,y):={\rm dist}(x,y)$ is the Riemannian distance between $x$ and $y$.

\item[(d)] There exists a constant $A$ such that $G_{y}(x)\geq A$. Since the Green function is defined up to a constant, we may choose one that is positive everywhere.

\item[(e)] $\int_{M}G_{y}(x)\,d\nu(x)$ is constant. We can choose $G_{y}(x)$ so that its integral equals zero.

\item[(f)] $G_{y}(x)=G_{x}(y)$.
\end{enumerate}
\end{thm}

In view of Theorem \ref{pg2}(e), we can choose a Green function $G_{y}(x)$ that is positive; moreover, there exists some constant $\widehat{C}_{0}$ such that for all $y\in M$,
\begin{equation}\label{C222}
\int_{M}G_{y}(x)\,d\nu(x)=\widehat{C}_{0}.
\end{equation}

Let $\mu$ be a positive finite Borel measure defined on $M$. We need the following condition: there exists some constant $C$ such that
\begin{equation}\label{C22}
\sup_{y\in M}\int_{M}G_{y}(x)\,d\mu(x)\leq C.
\end{equation}

A similar argument as that in the proof of \cite[Proposition 4.1]{Hu-Lau-Ngai_2006} implies that if $\underline{\dim}_{\infty}(\mu)>d-2$, then condition (\ref{C22}) holds. Using a similar argument as in Proposition \ref{prop3}, we can prove that if $\mu\in \big(\mathcal{W}(M)\big)'$, then there exists a unique weak solution $u\in\mathcal{W}(M)$ (see Proposition \ref{App2}). For $1\leq p\leq\infty$, We define the Green operator $G_{\mu}$ on $L^{p}(M,\mu)$ by
$$(G_{\mu}f)(x):=\int_{M}G_{y}(x)f(y)\,d\mu(y).$$
\begin{prop}\label{propLpp}
Assume the hypotheses of Theorem \ref{CCT}. Then there exists a constant $C>0$ such that for any $f\in L^{2}(M,\mu)$,
\begin{equation}\label{GuBB}
\|G_{\mu}f\|_{L^{p}(M,\mu)}\leq C\|f\|_{L^{p}(M,\mu)}
\end{equation}
and
\begin{equation}\label{GuBB2}
\|G_{\mu}f\|_{L^{p}(M)}\leq C\|f\|_{L^{p}(M,\mu)}.
\end{equation}
\end{prop}
\begin{proof}
The proof is the same as that of \cite[Proposition 4.2]{Hu-Lau-Ngai_2006}, and is omitted.
\end{proof}

\begin{prop}\label{rmk-1}
Assume the hypotheses of Proposition \ref{propLpp}. Let $f\in L^{2}(M,\mu)$. Then $G_{\mu}f\in W^{1,2}(M)$.
\end{prop}
\begin{proof}
Let $f\in L^{2}(M,\mu)$. We claim that $G_{\mu}f$ is a weak solution of
\begin{equation}\label{eqweak2}
\Delta u=f\mu.
\end{equation}
In fact, by (\ref{GuBB2}), $G_{\mu}f\in L^{2}(M)$. For any $\xi\in C_{c}^{\infty}(M)$ satisfying $\int_{M}\xi\,d\nu=0$, we have
\begin{align}\label{eqeqeq1}
\int_{M}\Big(&\int_{M}G_{y}(x)f(y)\,d\mu(y)\Big)\Delta\xi\,d\nu=\int_{M}\Big(\int_{M}G_{y}(x)\,d(f\mu)(y)\Big)\Delta\xi\,d\nu\notag\\
&=\int_{M}\int_{M}G_{y}(x)\Delta\xi(x)\,d\nu(x)d(f\mu)(y)\notag\\
&=\int_{M}\Big(\xi(y)-\nu(M)^{-1}\int_{M}\xi(x)\,d\nu(x)\Big)\,d(f\mu)(y)\qquad(\text{by~(\ref{deltaun})})\notag\\
&=\int_{M}\xi(y)\,d(f\mu)(y),\qquad\qquad\Big(\int_{M}\xi\,d\nu=0\Big)
\end{align}
which implies that $G_{\mu}f$ is a weak solution of (\ref{eqweak2}). Using similar methods as those in the proof of  Propositions \ref{prop3} and \ref{prop5}, one can prove that $f\mu\in\left(\mathcal{W}(M)\right)'$ (see Proposition \ref{App3}) and there exists a unique weak solution $u\in \mathcal{W}(M)$ of (\ref{2.1}) (see Proposition \ref{App2}). Thus, for all $\xi\in C_{c}^{\infty}(M)$ satisfying $\int_{M}\xi\,d\nu=0$, we have by (\ref{7.2})
\begin{equation}\label{smooth}
\int_{M}(u-G_{\mu}f)\Delta\xi\,d\nu=0.
\end{equation}
Note that for any $\phi\in C_{c}^{\infty}(M)$, $\zeta:=\phi-\nu(M)^{-1}\int_{M}\phi\,d\nu\in C_{c}^{\infty}(M)$ and $\int_{M}\zeta\,d\nu=0$. By (\ref{smooth}), we have
$$
\int_{M}(u-G_{\mu}f)\Delta\phi\,d\nu=\int_{M}(u-G_{\mu}f)\Delta\Big(\zeta+\nu(M)^{-1}\int_{M}\phi\,d\nu\Big)\,d\nu=0.$$
Hence, (\ref{smooth}) holds for any $\phi\in C_{c}^{\infty}(M)$. It follows by Weyl's lemma, $u-G_{\mu}f$ is smooth and harmonic. Since $M$ is compact, any harmonic function is constant. Thus, $u-G_{\mu}f=C$, a constant. The linearity of $W^{1,2}(M)$ and $u\in\mathcal{W}(M)\subseteq W^{1,2}(M)$ imply $G_{\mu}f\in W^{1,2}(M)$.
\end{proof}

Let $\widetilde{\mathcal{H}}_{\mu}:=\left\{u\in L^{2}(M,\mu):\int_{M}u\,d\mu=0\right\}$ and $D_{\mu}:=G_{\mu}(\widetilde{\mathcal{H}}_{\mu})$.

\begin{thm}\label{OG}
Assume the hypotheses of Theorem \ref{CCT}. Let $f\in \widetilde{\mathcal{H}}_{\mu}$. Then $G_{\mu}f\in {\rm dom}(\Delta_{\mu})$ and $\Delta_{\mu}(G_{\mu}f)=f.$ Consequently, $D_{\mu}\subseteq {\rm dom}(\Delta_{\mu})$ and $G_{\mu}|_{\widetilde{\mathcal{H}}_{\mu}}=\big(\Delta_{\mu}|_{ D_{\mu}}\big)^{-1}$.
\end{thm}
\begin{proof}
Let $f\in\widetilde{\mathcal{H}}_{\mu}$. We divide the proof into several steps.

\noindent{\em Step 1.} We prove $G_{\mu}f\in{\rm dom}(\mathcal{E})$. By Proposition \ref{rmk-1}, $G_{\mu}f\in W^{1,2}(M)$.
Moreover, let $u\in({\rm dom}(\mathcal{E}))^{\bot}$, which implies that $\|u\|_{L^{2}(M,\mu)}=0$ and there exists $u_{n}\in C_{c}^{\infty}(M)$ such that $u_{n}\to u$ in both $W^{1,2}(M)$ and $L^{2}(M,\mu)$. Thus,
\begin{align*}
\mathcal{E}(G_{\mu}f,u)&=\int_{M}\langle\nabla G_{\mu}f,\nabla u\rangle\,d\nu=\lim_{n\to\infty}\int_{M}\Big(\int_{M}G_{y}(x)f(y)\,d\mu(y)\Big)\Delta u_{n}(x)\,d\nu(x)\\
&=\lim_{n\to\infty}\int_{M}\Big(\int_{M}G_{y}(x)\Delta u_{n}(x)\,d\nu(x)\Big)f(y)\,d\mu(y)\qquad\qquad\quad(\text{Fubini})\\
&=\lim_{n\to\infty}\int_{M}\Big(u_{n}(y)-\frac{1}{\nu(M)}\int_{M}u_{n}(x)\,d\nu(x)\Big)f(y)\,d\mu(y)\qquad(\text{by}~(\ref{deltaun}))\\
&=\lim_{n\to\infty}\int_{M}u_{n}(y)f(y)\,d\mu(y).\qquad\qquad\qquad\qquad(f\in\widetilde{\mathcal{H}}_{\mu})
\end{align*}
Therefore,
$$\big|\mathcal{E}(G_{\mu}f,u)\big|\leq\lim_{n\to\infty}\|u_{n}\|_{L^{2}(M,\mu)}\|f\|_{L^{2}(M,\mu)}=0,$$
which implies that $G_{\mu}f\in{\rm dom}(\mathcal{E})$.

\noindent{\em Step 2.} We prove that $G_{\mu}f\in{\rm dom}(\Delta_{\mu})$ and $\Delta_{\mu}(G_{\mu}f)=f$. Since $G_{\mu}f\in{\rm dom}(\mathcal{E})$, by \cite[Proposition 3.3]{Ngai-Ouyang_2024}, to prove Step 2, it suffices to show $\Delta(G_{\mu}f)=f\,d\mu$ in the sense of distribution. For any $v\in C_{c}^{\infty}(M)$, by Fubini's theorem and (\ref{deltaun}), one can show that
$$\int_{M}v\Delta(G_{\mu}f)\,d\nu=\int_{M}(\Delta v)G_{\mu}f\,d\nu=\int_{M}fv\,d\mu,$$
proving that $\Delta(G_{\mu}f)=f\,d\mu$ in the sense of distribution. Therefore, Step 2 is proved.
Moreover, since $f\in\widetilde{H}_{\mu}$ is arbitrary, we have $D_{\mu}\subseteq{\rm dom}(\Delta_{\mu})$ and $\Delta_{\mu}(D_{\mu})=\widetilde{H}_{\mu}$.

\noindent{\em Step 3.} We prove that for any $u\in D_{\mu}$, $G_{\mu}(\Delta_{\mu}u)=u$. Let $u\in D_{\mu}=G_{\mu}(\widetilde{\mathcal{H}}_{\mu})$ be arbitrary. Then there exists $f\in\widetilde{\mathcal{H}}_{\mu}$ such that $G_{\mu}f=u$. Hence, by Step 2,
$$G_{\mu}(\Delta_{\mu}u)=G_{\mu}(\Delta_{\mu}(G_{\mu}f))=G_{\mu}f=u,$$
which proves Step 3. Combining Steps 2 and 3 proves that $G_{\mu}|_{\widetilde{\mathcal{H}}_{\mu}}=\big(\Delta_{\mu}|_{ D_{\mu}}\big)^{-1}$.
\end{proof}

\begin{rmk}
One cannot conclude that $G_{\mu}$ is the inverse of $-\Delta_{\mu}$ on $L^{2}(M,\mu)$, since for all constant functions $\phi=C$, $\Delta_{\mu}(G_{\mu}\phi)=G_{\mu}(\Delta_{\mu}\phi)=0$.
\end{rmk}

\begin{proof}[Proof of Theorem \ref{CCT}]
By Proposition \ref{prop>0}(a), a nonzero constant $C$ is an eigenfunction of $\Delta_{\mu}$ and is continuous. Therefore, we only consider nonconstant eigenfunctions. Note that each nonconstant eigenfunction $u$ belong to $ {\rm dom}(\Delta_{\mu})\cap\widetilde{\mathcal{H}}_{\mu}$. Let $f\in {\rm dom}(\Delta_{\mu})\cap\widetilde{\mathcal{H}}_{\mu}$, we claim that $f$ is a $\lambda$-eigenfunction of $\Delta_{\mu}$ if and only if $f=\lambda G_{\mu}f$. In fact, on the one hand, if $f\in {\rm dom}(\Delta_{\mu})\cap\widetilde{\mathcal{H}}_{\mu}$ satisfies $f=\lambda G_{\mu}f$, then by Theorem \ref{OG},
$$\Delta_{\mu}f=\lambda\Delta_{\mu}G_{\mu}f=\lambda f,$$
which implies that $f$ is a $\lambda$-eigenfunction of $\Delta_{\mu}$. On the other hand, if $f\in {\rm dom}(\Delta_{\mu})\cap\widetilde{\mathcal{H}}_{\mu}$ is a $\lambda$-eigenfunction, i.e., $\Delta_{\mu}f=\lambda f$, then
\begin{equation}\label{inverse}
G_{\mu}(\lambda f)=G_{\mu}(\Delta_{\mu}f).
\end{equation}
In view of Theorem \ref{OG}, $G_{\mu}|_{\widetilde{\mathcal{H}}_{\mu}}=\big(\Delta_{\mu}|_{ D_{\mu}}\big)^{-1}$. Hence, $G_{\mu}(\Delta_{\mu}f)=f.$ Combining this and (\ref{inverse}) implies that $f=\lambda G_{\mu}f$. Therefore, to prove Theorem \ref{CCT}, it suffices to prove that for any $f\in {\rm dom}(\Delta_{\mu})\cap\widetilde{H}_{\mu}$, $G_{\mu}f$ is continuous, but this follows by applying a similar argument as that in the proof of Proposition \ref{continuous}.
\end{proof}

\section{Continuous eigenfunctions on conformal Riemannian surfaces}

In this section, we let $M$ be a complete Riemann surface and $\Omega\subseteq M$ be a bounded domain which is conformally equivalent (see Definition \ref{CE}) to some bounded domain $\widetilde{\Omega}\subseteq\mathbb{R}^{2}$. Let $-\Delta_{\mu}$ be the Kre\u{\i}n-Feller defined on $\Omega$ as in Section 2. We give an alternative proof of the continuity of eigenfunctions of $-\Delta_{\mu}$ on $\Omega$. We also give an example of a continuous eigenfunction on a bounded domain of the sphere $\mathbb{S}^{2}$.

We first define conformal maps and the conformal equivalence of two Riemann surfaces. The following definitions can be found in, e.g., \cite{Freitag_2011}.

\begin{defi}
Let $X$ and $Y$ be two Riemann surfaces. A bijective map $f: X\to Y$ is called {\em conformal} if $f$ and $f^{-1}$ are both analytic.
\end{defi}

\begin{defi}\label{CE}
Two Riemann surfaces $X$ and $Y$ are called {\em conformally equivalent} if there exists a conformal map between $X$ and $Y$.
\end{defi}
Let $f:X\to Y$ be a conformal map, and $\mu$ be a positive finite Borel measure defined on $X$ with compact support ${\rm supp}(\mu)\subseteq\overline{\Omega}$, where $\Omega\subseteq X$ is a bounded domain. Let
\begin{equation}\label{eqmu}
\widetilde{\mu}:=\mu\circ f^{-1}
\end{equation}
be the positive finite Borel measure defined on $Y$ with compact support ${\rm supp}(\widetilde{\mu})\subseteq\overline{f(\Omega)}$. The proof of the following proposition is straightforward and is omitted.

\begin{prop}\label{prop-1}
Let $X$ and $Y$ be conformally equivalent Riemann surfaces with conformal map $f:X\to Y$. Let $\mu$ be a positive finite Borel measure defined on $X$ with compact support ${\rm supp}(\mu)\subseteq X$ and let $\widetilde{\mu}$ be defined as in (\ref{eqmu}) on $Y$. Assume $\underline{\rm dim}_{\infty}(\mu)>0$. Then $\underline{\rm dim}_{\infty}(\widetilde{\mu})>0$.
\end{prop}

Note that $\underline{\dim}_{\infty}(\mu)>0$ implies that $\Delta_{\mu}$ is well defined on some bounded domain of $M$.
\begin{thm}\label{th+}
Let $M_{2}$ be a complete smooth Riemann surface, $\Omega\subseteq M_{2}$ be a bounded domain which is conformally equivalent to a bounded domain $\widetilde{\Omega}\subseteq\mathbb{R}^{2}$ through a conformal map $f:\Omega\to \widetilde{\Omega}$ with  $|J_{f^{-1}}|$, the Jacobian determinant of $f^{-1}$, being bounded from above. Let $\mu$ be a positive finite Borel measure on $M_{2}$ with compact support ${\rm supp}(\mu)\subseteq\overline{\Omega}$ and with $\mu(\Omega)>0$. Let $\widetilde{\mu}$ be defined as in (\ref{eqmu}). Assume $\underline{\dim}_{\infty}(\mu)>0$. If we denote $\widetilde{u}$ be a $\lambda$-eigenfunction of $-\Delta_{\widetilde{\mu}}$ on $\widetilde{\Omega}$, then $u:=\widetilde{u}\circ f$ is a $\lambda$-eigenfunction of $-\Delta_{\mu}$ on $\Omega$. Consequently, $u$ is continuous on $\Omega$.
\end{thm}
\begin{proof}
Let $\Delta^{\Omega}$ and $\Delta^{\widetilde{\Omega}}$ be the Laplace-Beltrami operators defined on $\Omega$ and $\widetilde{\Omega}$ respectively. Let $G_{\widetilde{y}}^{\widetilde{\Omega}}(\widetilde{x})$ be the Green function of $\Delta^{\widetilde{\Omega}}$. By \cite[Theorem 1.3]{Hu-Lau-Ngai_2006}, $\widetilde{u}=G_{\widetilde{\mu}}(-\Delta_{\widetilde{\mu}}\widetilde{u})=\lambda G_{\widetilde{\mu}}\widetilde{u}.$
Hence, for any $\widetilde{x}\in\widetilde{\Omega}$,
\begin{equation}\label{eq-1}
\widetilde{u}(\widetilde{x})=\lambda \big(G_{\widetilde{\mu}}\widetilde{u}\big)(\widetilde{x})=\lambda\int_{\widetilde{\Omega}}G_{\widetilde{y}}^{\widetilde{\Omega}}(\widetilde{x})\cdot\widetilde{u}(\widetilde{y})\,d\widetilde{\mu}(\widetilde{y}).
\end{equation}
Let $x:=f^{-1}(\widetilde{x})$ and $y:=f^{-1}(\widetilde{y})$. By (\ref{eq-1}), we have
\begin{equation}\label{eq-2}
u(x)=\widetilde{u}(f(x))=\lambda\int_{\Omega}G_{f(y)}^{\widetilde{\Omega}}(f(x))\cdot\widetilde{u}(f(y))\,d(\widetilde{\mu}\circ f)(y).
\end{equation}
It view of \cite[Theorem 1]{Gutkin-Newton_2004}, which asserts that if two Riemann surfaces $Z$ and $U$ endowed with conformal metrics, and are conformal equivalent through $\phi:Z\to U$, then the Green's functions $G^{Z}_{y}(x)$ and $G^{U}_{y}(x)$ on $Z$ and $U$ respectively, satisfy $G^{Z}_{y}(x)=G^{U}_{\phi(y)}(\phi(x)).$ Moreover, every Riemann surface admits a metric which conformal to the Euclidean (see, e.g., \cite[Lemma 10.1.1]{Jost_2017}).  Therefore, $G_{y}^{\Omega}(x)=G_{f(y)}^{\widetilde{\Omega}}(f(x)).$ Combining this and (\ref{eq-2}), we have
\begin{equation}\label{eq-3}
u(x)=\lambda\int_{\Omega}G_{y}^{\Omega}(x)\cdot u(y)\,d\mu(y)=\lambda(G_{\mu}u)(x).
\end{equation}
Since $f$ is bijective, (\ref{eq-3}) holds for all $x\in\Omega$. Moreover, $u\in W_{0}^{1,2}(\Omega)$ as $\widetilde{u}\in W_{0}^{1,2}(\widetilde{\Omega})$ and $u\in L^{2}(\Omega,\mu)$ as $\widetilde{u}\in L^{2}(\widetilde{\Omega},\widetilde{\mu})$. Thus, by a same argument as in \cite[Theorem 1.3]{Hu-Lau-Ngai_2006}, we have $u\in{\rm dom}(-\Delta_{\mu})$ and $-\Delta_{\mu}u=-\Delta_{\mu}\big(\lambda G_{\mu}u\big)=\lambda u,$ i.e., $u$ is a $\lambda$-eigenfunction of $-\Delta_{\mu}$. In view of \cite[Theorem 1.2]{Ngai-Zhang-Zhao_2024}, $u$ is continuous on $\Omega$.
\end{proof}

Basis on Theorem \ref{th+} and \cite[Example 6.1]{Ngai-Zhang-Zhao_2024}, we have the following example. Let $\mathbb{S}^{2}\subseteq\mathbb{R}^{3}$ and $\mathbb{D}^{2}\subseteq\mathbb{R}^{2}$ be the sphere of radius $2$ and the open disk of radius $2$ respectively, i.e.,
$$\mathbb{S}^{2}:=\{(x_{1},x_{2},x_{3})\in\mathbb{R}^{3}:\sum_{i=1}^{3}x_{i}^{2}=4\}\quad\text{and}\quad\mathbb{D}^{2}:=\{(x_{1},x_{2})\in\mathbb{R}^{2}:x_{1}^{2}+x_{2}^{2}<4\}.$$
Let $\mathbb{S}^{2}_{+}:=\{(x_{1},x_{2},x_{3})\in \mathbb{S}^{2}:x_{3}>0\}$ be the upper hemisphere of $\mathbb{S}^{2}$. Recall that the {\em stereographic projection} $\phi:\mathbb{S}^{2}_{+}\to\mathbb{D}^{2}$ is defined as
$$\phi(x_{1},x_{2},x_{3})=\frac{2}{2+x_{3}}(x_{1},x_{2})=:(y_{1},y_{2}).$$
The inverse of $\phi$ is given by
$$\phi^{-1}(y_{1},y_{2})=\frac{2}{y_{1}^{2}+y^{2}_{2}+4}(4y_{1},4y_{2},4-y_{1}^{2}-y_{2}^{2}).$$
Obviously, $\phi$ is conformal. Hence, for any bounded domain $\Omega\subseteq\mathbb{S}^{2}_{+}$, there exists a bounded domain $\widetilde{\Omega}:=\phi(\Omega)\subseteq\mathbb{D}^{2}$ which is conformally equivalent to $\Omega$, and vice versa.

Let $\widetilde{U}:=(-1,1)\times(-1,1)\subseteq\mathbb{D}^{2}$ and let $U:=\phi^{-1}(U)$. Then $U$ and $\widetilde{U}$ are conformally equivalent. Let $\widetilde{\mu}_{0}$ and $\widetilde{\mu}_{1}$ be the 1-dimension Lebesgue measures defined on $(-1,1)\times\{0\}$ and $\{0\}\times(-1,1)$ respectively. Let $\widetilde{\mu}:=\widetilde{\mu}_{0}+\widetilde{\mu}_{1}$ (see Figure 1). Let $\mu:=\widetilde{\mu}\circ\phi$ on $U$ (see Figure 2). Obviously, $\underline{\rm dim}_{\infty}(\widetilde{\mu})=1>0$. In view of Proposition \ref{prop-1}, $\underline{\rm dim}_{\infty}(\mu)>0$. Let $\Delta_{\widetilde{\mu}}$ and $\Delta_{\mu}$ be the Kre\u{\i}n-Feller operators defined by $\widetilde{\mu}$ on $\widetilde{U}$ and $\mu$ on $U$ respectively. It follows by \cite[Example 6.1]{Ngai-Zhang-Zhao_2024} that $\widetilde{u}(x,y)=1+|xy|-|x|-|y|$ is a $2$-eigenfunction of $-\Delta_{\widetilde{\mu}}$ and $\widetilde{u}$ is continuous on $U$. Moreover, rewriting
$$\phi(x_{1},x_{2},x_{3})=\frac{2}{2+x_{3}}(x_{1},x_{2},0)=\Big(\frac{2x_{1}}{2+x_{3}},\frac{2x_{2}}{2+x_{3}},x_{1}^{2}+x_{2}^{2}+x_{3}^{2}-4\Big),$$
we have
$$|J_{\phi}|=
\left| \begin{array}{ccc}
\frac{2}{2+x_{3}} & 0 & \frac{-2x_{1}}{(2+x_{3})^{2}} \\
0 & \frac{2}{2+x_{3}} & \frac{-2x_{2}}{(2+x_{3})^{2}} \\
2x_{1} & 2x_{2} & 2x_{3}
\end{array} \right|=\frac{16}{(2+x_{3})^{2}}.
$$
Therefore, $|J_{\phi^{-1}}|=\frac{(2+x_{3})^{2}}{16}\leq1$. Hence, the assertions in the following example follow directly from Theorem \ref{th+}.
\begin{figure}[htbp]
\centering
\begin{minipage}
  {0.45\textwidth}
  \centering
  \begin{tikzpicture}[scale=0.6,domain=0:180,>=stealth]
 \coordinate (org) at (0,0);
    \foreach\i/\text in{{0,5}/{},{5,0}/{x}}
    \draw[help lines,->] (org)node[above right]{}--(\i)node[above]{$\text$};
    \draw[help lines,->]node[left] at (0,5){$y$};
    \draw[help lines](0,-5)--(0,0);
    \draw[help lines](-5,0)--(0,0);
    \draw[thin,black,dashed] (-3,3)--(3,3);
    \draw[thin,black,dashed] (-3,-3)--(-3,3);
    \draw[thin,black,dashed] (3,3)--(3,-3);
    \draw[thin,black,dashed] (-3,-3)--(3,-3);

 \draw[thick,red](3,0)--(-3,0);
   \node[above,font=\tiny] at (1.8,0){\color{red}$\widetilde{\mu}_{0}$};
   \draw[help lines,->](1.6,0.3)--(1.6,0);
  \draw[thick,green](0,3)--(0,-3);
  \node[left,font=\tiny] at (0,-1.8){\color{green}$\widetilde{\mu}_{1}$};
  \draw[help lines,->](-0.3,-1.8)--(0,-1.8);
  \node[below, font=\tiny,text=gray] at (0,-3) {$-1$};
  \node[below, font=\tiny,text=gray] at (3.2,0) {$1$};
   \node[left, font=\tiny,text=gray] at (0,3.4) {$1$};
  \node[below, font=\tiny,text=gray] at (-3.6,0) {$-1$};
   \end{tikzpicture}
 \caption{The measure $\widetilde{\mu}$}
 \end{minipage}
 \hspace{0.05\textwidth}
 \begin{minipage}
 {0.45\textwidth}
 \centering
\begin{tikzpicture}[scale=.3,domain=0:180,>=stealth]
    \coordinate (org) at (0,0);
    \draw (0,0) circle[radius=8];
    \draw (org) ellipse (8cm and 3cm);
    \draw[help lines,dashed](9.5,0)--(0,0);
    \draw[help lines,dashed](0,-9.5)--(0,0);
    \draw[help lines,dashed](5.8,7.25)--(0,0);
    \draw[white] plot ({8*cos(\x)},{3*sin(\x)});
    \foreach\i/\text in{{-6.8,-8.5}/y,{-10.5,0}/x,{0,10.5}/{}}
    \draw[help lines,->] (org)node[above right]{}--(\i)node[above]{$\text$};
    \draw[help lines,->]node[left] at (0,10){$z$};
    \draw[thin,black] (-2.90,1.80)--(5.20,1.80);
    \draw[thin,black] (-5.52,-1.8)--(2.52,-1.8);
    \draw[thin,black] (5.20,1.80)--(2.5,-1.8);
    \draw[thin,black] (-2.90,1.80)--(-5.5,-1.8);
    \draw[thin,blue] (-6.5,4.7) to [out=-120, in=110](-7.1,0.2) to [out=36, in=156] (3.5,0.6) to [out=80, in=-140] (6.2,5);
    \draw[thin,blue,densely dashed] (-6.5,4.7) to [out=30, in=140] (-3.5,3.5) to [out=36,in=156] (6,3.2) to [out=70,in=90] (6.2,4.9);
    \draw[thin,cyan,densely dashed](0,-8)--(6.2,5);
    \draw[thin,cyan,densely dashed](0,-8)--(-6.5,4.7);
    \draw[thick,red](3.8,0)--(-4.2,0);
    \draw[thick,green](1.36,1.8)--(-1.36,-1.8);
    \draw[thick,red] (6.2,5.03) to [out=130, in=-3](0,8) to [out=182, in=50](-6.5,4.7);
    \draw[thick,green](-2.36,1.97) to [out=100, in=-170] (0,8);
    \draw[thin,cyan,densely dashed](0,-8)--(-2.36,1.97);
     \draw[thin,cyan,densely dashed](0,-8)--(1.8,4.6);
    \draw[thick,green,densely dashed](0,8) to [out=-5, in=90] (1.8,4.6);
 \end{tikzpicture}
\caption{The measure $\mu$}
\end{minipage}
\end{figure}
\begin{exam}\label{8.1}
Use the above notation. Then $u:=\widetilde{u}\circ\phi$ is a $2$-eigenfunction of $-\Delta_{\mu}$ on $U$ (see Figure 3). Consequently, $u\in C(U)$.
\end{exam}

\begin{figure}[h!]
\centering

\begin{minipage}{0.7\linewidth}
\centering
  \includegraphics[scale=0.6,width=0.6\linewidth]{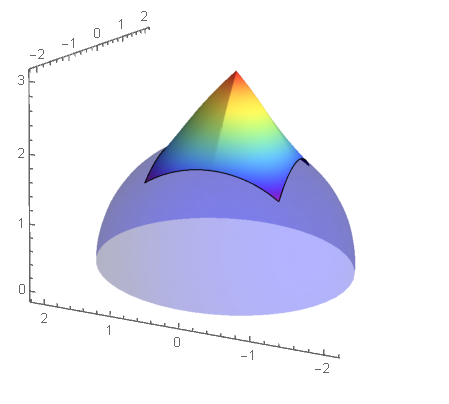}
\caption{The $2$-eigenfunction in Example \ref{8.1}.}
\end{minipage}
\end{figure}

\appendix
\section{Proof of Corollary \ref{corollary1}}
\renewcommand\theequation{A.\arabic{equation}}
\begin{proof}[Proof of Corollary \ref{corollary1}]
{\em Case 1. $n=1$.} Let $v\in C_{c}^{\infty}(\Omega)$ and $\epsilon>0$ be arbitrary. Then $u+\epsilon v\in\mathcal{H}_{\mu}$. Since $R_{\mu}(u)=\lambda_{1}$, by Lemma \ref{lemRu}, the function
$$f(\epsilon):=\frac{\int_{\Omega}|\nabla(u+\epsilon v)|^{2}\,d\nu}{\int_{\Omega}|u+\epsilon v|^{2}\,d\mu}$$
has a minimum at $\epsilon=0$. Hence $f'(0)=0$. Note that
$$f'(0)=\frac{\int_{\Omega}|u|^{2}\,d\mu\cdot2\int_{\Omega}\langle\nabla u,\nabla v\rangle\,d\nu-2\int_{\Omega}uv\,d\mu\cdot\int_{\Omega}|\nabla u|^{2}\,d\nu}{\Big(\int_{\Omega}|u|^{2}\,d\mu\Big)^{2}}.$$
Thus
\begin{equation}\label{eqeq}
\int_{\Omega}|u|^{2}\,d\mu\cdot\int_{\Omega}\langle\nabla u, \nabla v\rangle\,d\nu=\int_{\Omega}uv\,d\mu\cdot\int_{\Omega}|\nabla u|^{2}\,d\nu.
\end{equation}
Observing that
$$R_{\mu}(u)=\frac{\mathcal{E}(u,u)}{(u,u)_{L^{2}(\Omega,\mu)}}=\frac{\int_{\Omega}|\nabla u|^{2}\,d\nu}{\int_{\Omega}|u|^{2}\,d\mu}=\lambda_{1}.$$
Substitute this into (\ref{eqeq}), we have
$$\int_{\Omega}\langle\nabla u, \nabla v\rangle\,d\nu=\lambda_{1}\int_{\Omega}uv\,d\mu,$$
which implies that $u$ is a $\lambda_{1}$-eigenfunction.

\noindent {\em Case 2. $n\geq2$.} Let $v\in C_{c}^{\infty}(\Omega)$ and
$$\omega(x):=v(x)-\sum_{k=1}^{n-1}a_{k}u_{k}(x), \qquad\text{where}~a_{k}:=\frac{(v,u_{k})_{L^{2}(\Omega,\mu)}}{(u_{k},u_{k})_{L^{2}(\Omega,\mu)}}.$$
Then $\omega\in\mathcal{H}_{\mu}$ and for $k=1,\ldots,n-1$, $(\omega,u_{k})_{L^{2}(\Omega,\mu)}=0$.
Let
$$\widetilde{f}(\epsilon):=\frac{\int_{\Omega}|\nabla(u+\epsilon\omega)|^{2}\,d\nu}{\int_{\Omega}|u+\epsilon\omega|^{2}\,d\mu}.$$
Using the argument in Case 1 and replacing $v$ in (\ref{eqeq}) by $\omega$, we have
$$\int_{\Omega}|u|^{2}\,d\mu\cdot\int_{\Omega}\langle\nabla u, \nabla \omega\rangle\,d\nu=\int_{\Omega}u\omega\,d\mu\cdot\int_{\Omega}|\nabla u|^{2}\,d\nu,$$
i.e.,
\begin{equation}\label{eqeq1}
\int_{\Omega}\langle\nabla u, \nabla(v-\sum_{k=1}^{n-1}a_{k}u_{k})\rangle\,d\nu=\lambda_{n}\int_{\Omega}u(v-\sum_{k=1}^{n-1}a_{k}u_{k})\,d\mu,
\end{equation}
as $R_{\mu}(u)=\lambda_{n}$. Note that by Lemma \ref{lemRu}, for $k=1,\ldots,n-1$, $u\in({\rm span}\{u_{1},\ldots,u_{k}\})^{\bot}$.
Hence
\begin{equation}\label{eqeq2}
\mathcal{E}(u,u_{k})=(u,\lambda_{k}u_{k})_{L^{2}(\Omega,\mu)}=0.
\end{equation}
Combining (\ref{eqeq1}) and (\ref{eqeq2}), we have
$$\int_{\Omega}\langle\nabla u, \nabla v\rangle\,d\nu=\lambda_{n}\int_{\Omega}uv\,d\mu,$$
which implies that $u$ is a $\lambda_{n}$-eigenfunction.
\end{proof}

\section{Several propositions}
\setcounter{equation}{0}
\renewcommand\theequation{B.\arabic{equation}}
\renewcommand{\theprop}{B.\arabic{prop}}

Let $M$ be a compact connected smooth Riemannian manifold with $\partial M=\emptyset$. We first point that $W_{0}^{1,2}(M)=W^{1,2}(M)$ as $M$ is complete (see, e.g., \cite[Section 2.1]{Hebey_1996}) and $C^{\infty}(M)=C_{c}^{\infty}(M)$ as $M$ is compact.
Let $C>0$ be a constant and {\rm span}$\{C\}$ be the subspace of Hilbert space $W^{1,2}(M)$ spanned by $C$. Let
\begin{equation}\label{WM}
\mathcal{W}(M):={\rm span}\{C\}^{\bot}=\{f\in W^{1,2}(M):\int_{M}f\,d\nu=0\}.
\end{equation}
Obviously, $\mathcal{W}(M)$ is a closed subspace of $W^{1,2}(M)$. Hence $\mathcal{W}(M)$ is a Hilbert space (see e.g., \cite[Theorem 3.2--4]{Kreyszig_1978}). Moreover, there exists a constant $C>0$ such that for any $f\in\mathcal{W}(M)$ the following Poincar\'{e} inequality holds (see e.g., \cite[Section 3.1]{Schoen-Yau_1994})
\begin{equation}\label{PIMm}
\int_{M}|f|^{2}\,d\nu\leq C\int_{M}|\nabla f|^{2}\,d\nu.
\end{equation}
Note that (\ref{PIMm}) implies that $\mathcal{W}(M)$ is equipped with the equivalent norm $\|f\|_{\mathcal{W}(M)}=\|\nabla f\|_{L^{2}(M)}.$
\begin{prop}\label{propapp1}
Let $\mathcal{W}(M)$ be defined as in (\ref{WM}). Then $C_{c}^{\infty}(M)\cap\mathcal{W}(M)$ is dense in $\mathcal{W}(M)$.
\end{prop}
\begin{proof}
Let $u\in\mathcal{W}(M)\subseteq W^{1,2}(M)$ be arbitrary. Since $C_{c}^{\infty}(M)$ is dense in $W^{1,2}(M)$, there exists a sequence $\{\phi_{n}\}\subseteq C_{c}^{\infty}(M)$ such that $\phi_{n}\to u$ in $W^{1,2}(M)$ as $n\to\infty$. Let
$$v_{n}:=\phi_{n}-\frac{1}{\nu(M)}\int_{M}\phi_{n}\,d\nu.$$
Then $v_{n}\in C_{c}^{\infty}(M)$ and $\int_{M}v_{n}\,d\nu=0$. Thus $v_{n}\in C_{c}^{\infty}(M)\cap\mathcal{W}(M)$. Moreover,
\begin{align}\label{dense}
\|v_{n}-u\|_{W^{1,2}(M)}=&\Big\|\phi_{n}-\nu(M)^{-1}\int_{M}\phi_{n}\,d\nu-u\Big\|_{W^{1,2}(M)}\notag\\
=&\Big\|\phi_{n}-u-\nu(M)^{-1}\int_{M}(\phi_{n}-u)\,d\nu\Big\|_{W^{1,2}(M)}\qquad\quad(u\in\mathcal{W}(M))\notag\\
\leq&\|\phi_{n}-u\|_{W^{1,2}(M)}+\nu(M)^{-1/2}\int_{M}|\phi_{n}-u|\,d\nu\notag\\
\leq&2\|\phi_{n}-u\|_{W^{1,2}(M)}.\qquad\qquad\qquad(\text{by~H\"{o}lder's~inequality})
\end{align}
Letting $n\to\infty$ in both side of (\ref{dense}), we have $v_{n}\to u$ in $\mathcal{W}(M)$. Combining this and the fact that $C_{c}^{\infty}(M)\cap\mathcal{W}(M)\subseteq\mathcal{W}(M)$ completes the proof.
\end{proof}
Let $\big(\mathcal{W}(M)\big)'$ be the dual space of $\mathcal{W}(M)$. The following proposition follows from Proposition \ref{prop3}.

\begin{prop}\label{App2}
Let $d$ and $M$ be as in Theorem \ref{CCT}. If $\mu\in\big(\mathcal{W}(M)\big)'$, then there exists a unique weak solution $u\in \mathcal{W}(M)$ of (\ref{2.1}).
\end{prop}
\begin{proof}
By a similar argument as that in the proof of Proposition \ref{prop3}, one can prove that $u$ is a weak solution of (\ref{2.1}). Since $C_{c}^{\infty}(M)\cap\mathcal{W}(M)$ is dense in $\mathcal{W}(M)$, $u$ is unique.
\end{proof}
\begin{prop}\label{App3}
Assume the hypotheses of Theorem \ref{CCT}. Then for any $f\in L^{2}(M,\mu)$, $f\mu\in\big(\mathcal{W}(M)\big)'$.
\end{prop}
\begin{proof}
The proof is the same as that in the proof of Proposition \ref{prop5} and is omitted.
\end{proof}

\setcounter{equation}{0}

\end{document}